\documentclass[12pt]{amsart}
\usepackage{}

\usepackage{amsmath}
\usepackage{amsfonts}
\usepackage{amssymb}

\usepackage[all,cmtip]{xy}           
\usepackage{bm}
\usepackage{bbm}

\usepackage{bbding}
\usepackage{txfonts}
\usepackage{amscd}
\usepackage[all]{xy}
\usepackage[shortlabels]{enumitem}
\usepackage{ifpdf}
\ifpdf
  \usepackage[colorlinks,final,backref=page,hyperindex]{hyperref}
\else
  \usepackage[colorlinks,final,backref=page,hyperindex,hypertex]{hyperref}
\fi
\usepackage{tikz}
\usepackage[active]{srcltx}
\usepackage{caption}

\makeatletter

\newtheorem{thm}{Theorem}[section]
\newtheorem{lem}[thm]{Lemma}
\newtheorem{cor}[thm]{Corollary}
\newtheorem{pro}[thm]{Proposition}
\newtheorem{ex}[thm]{Example}
\newtheorem{rmk}[thm]{Remark}
\newtheorem{defi}[thm]{Definition}

\newcommand {\emptycomment}[1]{}

\newcommand{\be }{\begin{equation}}
\newcommand{\ee }{\end{equation}}

\newcommand{\huaL}{\mathcal{L}}
\newcommand{\huaR}{\mathcal{R}}

\newcommand{\huaP}{\mathcal{P}}

\newcommand{\huaJ}{\mathcal{J}}

\newcommand{\frkd}{\mathfrak d}

\newcommand{\frkr}{\mathfrak r}

\newcommand{\Id}{{\rm{Id}}}

\newcommand{\br}[1]{   [ \cdot,    \cdot  ]   }

\newcommand{\gl}{\mathfrak {gl}}

\newcommand{\Img}{\mathrm{Im}}

\begin{document}

\title{Quasi-triangular and factorizable dendriform D-bialgebras}

\author{You Wang}
\address{Department of Mathematics, Jilin University, Changchun 130012, Jilin, China}
\email{wangyou20@mails.jlu.edu.cn}

\vspace{-5mm}


\begin{abstract}
In this paper, we introduce the notions of quasi-triangular and factorizable dendriform D-bialgebras. A factorizable dendriform D-bialgebra leads to a factorization of the underlying dendriform algebra. We show that the dendriform double of a dendriform D-bialgebra naturally enjoys a factorizable dendriform D-bialgebra structure. Moreover, we introduce the notion of relative Rota-Baxter operators of nonzero weights on dendriform algebras and find that every quasi-triangular dendriform D-bialgebra can give rise to a relative Rota-Baxter operator of weight $1$. Then we introduce the notion of quadratic Rota-Baxter dendriform algebras as the Rota-Baxter characterization of factorizable dendriform D-bialgebras, and show that there is a one-to-one correspondence between factorizable dendriform D-bialgebras and quadratic Rota-Baxter dendriform algebras. Finally, we show that a quadratic Rota-Baxter dendriform algebra can give rise to an isomorphism from the regular representation to the coregular representation of a Rota-Baxter dendriform algebra.

\end{abstract}



\maketitle

\tableofcontents

\allowdisplaybreaks


\section{Introduction}


In this paper, we intend to establish the quasi-triangular and factorizable theories for dendriform D-bialgebras. Meanwhile, we also establish connections between (relative) Rota-Baxter operators and dendriform D-bialgebras.

Dendriform algebras were introduced by Loday \cite{Loday1} with motivation from algebraic K-theory, in order to establish relations with several areas in mathematics and physics, including homology \cite{Fra1,Fra2}, Hopf algebras \cite{Holtkamp,LR2,Ronco}, Lie and Leibniz algebras \cite{Fra2}, Rota-Baxter systems \cite{Brz}, combinatorics \cite{LR}, arithmetic \cite{Loday2} and so on. In \cite{BGN2012}, the notion of $\mathcal{O}-$operator was introduced (such a structure appeared independently in \cite{Uchino08} named after the generalized Rota-Baxter operator, also as the name of relative Rota-Baxter operator in \cite{LST}) and an $\mathcal{O}-$operator associated to a bimodule of an associative algebra can give rise to a dendriform algebra.

The notions of dendriform bialgebras can be traced back to the earlier researches \cite{Aguiar3,LR,LR2,Ronco}. There are three algebraic structures closely related to dendriform bialgebras, which are dendriform Hopf algebras, bidendriform bialgebras and dendriform D-bialgebras.
In \cite{Ronco}, the author proves a Milnor-Moore style theorem for dendriform Hopf algebras. It is shown that a dendriform Hopf algebra is isomorphic to the enveloping algebra of its brace algebra of primitive elements, used by the construction of Eulerian idempotents. In \cite{Foissy}, the notion of bidendriform bialgebras was introduced as a special dendriform bialgebras. For example, the Malvenuto-Reutenauer Hopf algebra and the non-commutative Connes-Kreimer Hopf algebras of planar decorated rooted trees are bidendriform bialgebras. Then the notion of dendriform D-bialgebras \cite{Bai} was introduced by Bai in order to establish connections with the double constructions of Connes cocycles and matched pairs of associative algebras. All these bialgebras are dendriform algebras equipped with coassociative operations satisfying some different compatibility conditions. However, it is obvious that dendriform D-bialgebras are quite different from the formers according to the definitions.

In the context of Lie bialgebras, quasi-triangular Lie bialgebras play important roles in mathematical physics \cite{Drinfeld,STS}. Moreover, as an important class of quasi-triangular Lie bialgebras, factorizable Lie bialgebras can be used to connect classical r-matrices with certain factorization problems and have widespread applications in integrable systems \cite{BGN2010,RS,S2}.

Note that even though dendriform bialgebras were widely studied recently and the coboundary dendriform D-bialgebras was given by Bai in \cite{Bai}, while the corresponding quasi-triangular, and in particular factorizable theories, are yet to be found. In this paper, we introduce an $(L_\succ,R_\prec)$ condition of a $2$-tensor, in order to define quasi-triangular dendriform D-bialgebras successfully. In particular, if the skew-symmetric part of the dendriform r-matrix (or a solution of the D-equation equivalently) is nondegenerate, then we obtain a factorizable dendriform D-bialgebra. We show that a factorizable dendriform D-bialgebra leads to a factorization of the underlying dendriform algebra. Furthermore, the dendriform double of a dendriform D-bialgebra has a natural factorizable dendriform D-bialgebra structure as an example.

There is a close relationship among Lie algebras, associative algebras, pre-Lie algebras and dendriform algebras as follows (in the sense of commutative diagram of categories). See \cite{Aguiar1,Aguiar3,Fard} for more details.
\[
\xymatrix{
  &{\rm dendriform~algebras}~(A,\prec,\succ) \ar[d]_{x\ast y=x\prec y+x\succ y} \ar[r]^{x\star y=x\succ y-y\prec x}  & {\rm pre-Lie~algebras}~(A,\star) \ar[d]^{[x,y]=x\star y-y\star x}   \\
  & {\rm associative~algebras}~(A,\ast)  \ar[r]^{[x,y]=x\ast y-y\ast x}  &  {\rm Lie~algebras}~ (A,[\cdot,\cdot])      }
\]
Bai extends the above relationships at the level of bialgebras with dualities in a commutative diagram in \cite{Bai}. In this paper, we intend to add Rota-Baxter operators and certain nondegenerate bilinear forms to among the four algebras, to get the following commutative diagram.
\[
\xymatrix{
  &{\rm quadratic~RB~dendriform~algebras} \ar[d]_{} \ar[r]^{}  & {\rm quadratic~RB~pre-Lie~algebras} \ar[d]^{}   \\
  & {\rm RB~associative~algebras~with~Connes~cocycles} \ar[u]_{} \ar[r]^{}  &  {\rm symplectic~RB~Lie~algebras}  \ar[u]_{}    }
\]
More interestingly, the left vertical arrows implies that there is a one-to-one correspondence between quadratic Rota-Baxter dendriform algebras and Rota-Baxter associative algebras with nondegenerate Connes cocycles. Recently, it was shown that factorizable Lie bialgebras can be characterized by quadratic Rota-Baxter Lie algebras of nonzero weights \cite{Lang}. Fortunately, we show that there is a one-to-one correspondence between factorizable dendriform D-bialgebras and quadratic Rota-Baxter dendriform algebras, which implies that quadratic Rota-Baxter dendriform algebras are the Rota-Baxter characterization of factorizable dendriform D-bialgebras.

The paper is organized as follows. In Section \ref{quasi-fac}, we introduce the notion of quasi-triangular dendriform D-bialgebras as a special class of coboundary dendriform D-bialgebras. Then we introduce the notion of a factorizable dendriform D-bialgebra, which is a special quasi-triangular dendriform D-bialgebra. We show that a factorizable dendriform D-bialgebra leads to a factorization of the underlying dendriform algebra. We also show that the dendriform double of an arbitrary dendriform D-bialgebra naturally has a factorizable dendriform D-bialgebra structure. In Section \ref{nonzero-rRBo}, we introduce the notion of relative Rota-Baxter operators of nonzero weights on dendriform algebras with respect to the coregular representation and find that  quasi-triangular dendriform D-bialgebras can give rise to two classes of relative Rota-Baxter operator of weight $1$, denoted by $r_+$ and $r_-$ respectively. In Section \ref{QRBD}, we give the Rota-Baxter characterization of factorizable dendriform D-bialgebras. We introduce the notion of quadratic Rota-Baxter dendriform algebras and show that there is a one-to-one correspondence between quadratic Rota-Baxter dendriform algebras and Rota-Baxter associative algebras with nondegenerate Connes cocycles. Moreover, we also show that there is a one-to-one correspondence between factorizable dendriform D-bialgebras and quadratic Rota-Baxter dendriform algebras. At the end of this section, we also show that a quadratic Rota-Baxter dendriform algebra can give rise to an isomorphism from the regular representation to the coregular representation of a Rota-Baxter dendriform algebra.

In this paper, we work over a field $\mathbb{K}$ and all the vector spaces and algebras are over $\mathbb{K}$ and finite-dimensional.

\vspace{2mm}
\noindent
{\bf Acknowledgements.} We give our warmest thanks to Yunhe Sheng and Chengming Bai for very useful discussions.

\newpage

\section{Quasi-triangular dendriform D-bialgebras and factorizable dendriform D-bialgebras}\label{quasi-fac}

In this section, we introduce the notions of quasi-triangular dendriform D-bialgebras and factorizable dendriform D-bialgebras. We show that the double of a dendriform D-bialgebra is naturally a factorizable dendriform D-bialgebra.

We first recall the definition of dendriform algebras.
\begin{defi}
Let A be a vector space with two bilinear products denoted by $\prec_A$ and $\succ_A$. $(A,\prec_A,\succ_A)$ is called a {\bf dendriform algebra} if for any $x,y,z\in A,$
\begin{eqnarray*}\label{dendriform}
(x\prec_A y)\prec_A z&=&x\prec_A(y\prec_A z+y\succ_A z),\\
(x\succ_A y)\prec_A z&=&x\succ_A(y\prec_A z),\\
x\succ_A (y \succ_A z)&=&(x\prec_A y+x\succ_A y)\succ_A z.
\end{eqnarray*}
\end{defi}

Let $(A,\prec_A,\succ_A)$ be a dendriform algebra. Then the new multiplication $\ast$ given by
\begin{eqnarray}\label{ast}
x \ast y=x\prec_A y+x\succ_A y, \quad \forall x,y\in A,
\end{eqnarray}
defines an associative algebra $(A,\ast)$, which is called the sub-adjacent associative algebra of $(A,\prec_A,\succ_A)$. Furthermore, $(A;L_{\succ_A},R_{\prec_A})$ gives a representation of the associative algebra $(A,\ast)$, where $L_{\succ_A},R_{\prec_A}:A\longrightarrow \gl(A)$ are defined by $L_{\succ_A}(x)y=x \succ_A y, R_{\prec_A}(x)y=y\prec_A x$ respectively.

\begin{defi}{\rm(\cite{Aguiar3})}
A representation of a dendriform algebra $(A,\prec_A,\succ_A)$ is a quintuple $(V;l_\succ,r_\succ,$ $l_\prec,r_\prec)$, where $V$ is a vector space and $l_\succ,r_\succ,l_\prec,r_\prec:A\longrightarrow \gl(V)$ are four linear maps satisfying the following equations:
\begin{equation*}
\left\{ \begin{aligned}
l_\prec(x\prec_A y)&=l_\prec(x)l_\ast(y);&\quad r_\prec(x)l_\prec(y)&=l_\prec(y)r_\ast(x);&\quad r_\prec(x)r_\prec(y)&=r_\prec(y\ast x);\\
l_\prec(x\succ_A y)&=l_\succ(x)l_\prec(y);&\quad r_\prec(x)l_\succ(y)&=l_\succ(y)r_\prec(x);&\quad r_\prec(x)r_\succ(y)&=r_\succ(y\prec_A x);\\
l_\succ(x\ast y)&=l_\succ(x)l_\succ(y);&\quad r_\succ(x)l_\ast(y)&=l_\succ(y)r_\succ(x);&\quad r_\succ(x)r_\ast(y)&=r_\succ(y\succ_A x),
\end{aligned}\right.
\end{equation*}
where $x \ast y=x\prec_A y+x\succ_A y,~ l_{\ast}=l_{\prec}+l_{\succ},~ r_{\ast}=r_{\prec}+r_{\succ}.$
\end{defi}

In fact, $(V;l_\succ,r_\succ,l_\prec,r_\prec)$ is a representation of a dendriform algebra $(A,\prec_A,\succ_A)$ if and only if the direct sum $A\oplus V$ of vector spaces is a dendriform algebra
(the semi-direct product) where the multiplications $\succ_{A\oplus V},\prec_{A\oplus V}$ on $A\oplus V$ are defined by
\begin{eqnarray*}
(x+u)\succ_{A\oplus V}(y+v)&=&x\succ_A y+l_\succ(x)v+r_\succ(y)u,\\
(x+u)\prec_{A\oplus V}(y+v)&=&x\prec_A y+l_\prec(x)v+r_\prec(y)u,
\end{eqnarray*}
for all $x,y\in A,u,v\in V.$ We denote it by $A \ltimes_{l_\succ,r_\succ,l_\prec,r_\prec} V$ or simply by $A \ltimes V$.

\begin{ex}\label{re-core-rep}
Let $(A,\prec_A,\succ_A)$ be a dendriform algebra. Then $(A;L_{\succ_A},R_{\succ_A},L_{\prec_A},R_{\prec_A})$ is a representation of $(A,\prec_A,\succ_A)$, which is called the {\bf regular representation}, where $R_{\succ_A},L_{\prec_A}:A\longrightarrow \gl(A)$ are defined by $R_{\succ_A}(x)y=y\succ_A x,L_{\prec_A}(x)y=x \prec_A y$ respectively for all $x,y\in A$. Define four linear maps $L_{\succ_A}^*,R_{\succ_A}^*,L_{\prec_A}^*,R_{\prec_A}^*:A\longrightarrow \gl(A^*)$ respectively by
\begin{eqnarray*}
&&\langle L_{\succ_A}^*(x)\xi,y  \rangle=\langle \xi,L_{\succ_A}(x)y \rangle, \quad \langle R_{\succ_A}^*(x)\xi,y  \rangle=\langle \xi,R_{\succ_A}(x)y \rangle,\\
&&\langle L_{\prec_A}^*(x)\xi,y  \rangle=\langle \xi,L_{\prec_A}(x)y \rangle, \quad \langle R_{\prec_A}^*(x)\xi,y  \rangle=\langle \xi,R_{\prec_A}(x)y \rangle,
\end{eqnarray*}
for all $x,y\in A, \xi\in A^*$. Then $(A^*;R_{\succ_A}^*+R_{\prec_A}^*,-L_{\prec_A}^*,-R_{\succ_A}^*,L_{\succ_A}^*+L_{\prec_A}^*)$ is also a representation of $(A,\prec_A,\succ_A)$, which is called the {\bf coregular representation}.
\end{ex}

\subsection{Quasi-triangular dendriform D-bialgebras}
In this subsection, we first recall dendriform D-bialgebras and then introduce the notion of quasi-triangular dendriform D-bialgebras.

\begin{defi}{\rm(\cite{Bai})}
Let $A$ be a vector space. A dendriform D-bialgebra structure on $A$ consists of the following data:
\begin{itemize}
\item[{\rm(i)}] $\Delta_\prec,\Delta_\succ:A\longrightarrow A\otimes A$ are two linear maps such that
$(\prec_{A^*}:=\Delta_\prec^*,\succ_{A^*}:=\Delta_\succ^*):A^*\otimes A^* \longrightarrow A^*$ defines a dendriform algebra structure on $A^*$;
\item[{\rm(ii)}] $\beta_\prec,\beta_\succ:A^*\longrightarrow A^*\otimes A^*$ are two linear maps such that $(\prec_{A}:=\beta_\prec^*,\succ_{A}:=\beta_\succ^*):A\otimes A \longrightarrow A$ defines a dendriform algebra structure on $A$;
\item[{\rm(iii)}] $\Delta_\prec,\Delta_\succ,\beta_\prec,\beta_\succ$ satisfy the following equilities:
\begin{eqnarray}
\label{equ1}\Delta_\prec(x\prec_{A}y+x\succ_{A}y)-\big(\Id\otimes L_{\succ _A}(x)\big)\Delta_\prec(y)-(\huaR_A(y)\otimes \Id)\Delta_\prec(x)&=&0;\\
\label{equ2}\Delta_\succ(x\prec_{A}y+x\succ_{A}y)-\big(\Id\otimes \huaL_{A}(x)\big)\Delta_\succ(y)-(R_{\prec _A}(y)\otimes \Id)\Delta_\succ(x)&=&0;\\
\label{equ3}\beta_\prec(\xi\prec_{A^*}\eta+\xi\succ_{A^*}\eta)-\big( \Id\otimes L_{\succ_{A^*}}(\xi)\big) \beta_\prec(\eta)-(\huaR_{A^*}(\eta)\otimes \Id)\beta_\prec(\xi)&=&0;\\
\label{equ4}\beta_\succ(\xi\prec_{A^*}\eta+\xi\succ_{A^*}\eta)-\big( \Id\otimes \huaL_{A^*}(\xi)\big) \beta_\succ(\eta)-(R_{\prec_{A^*}}(\eta)\otimes \Id)\beta_\succ(\xi)&=&0;\\
\label{equ5}\big( \huaL_A (x)\otimes \Id-\Id \otimes R_{\prec_A}(x) \big)\Delta_\prec(y)+\sigma \Big( \big( L_{\succ_A} (y)\otimes \Id-\Id \otimes \huaR_{A}(y) \big) \Delta_\succ(x) \Big)&=&0;\\
\label{equ6}\big( \huaL_{A^*} (\xi)\otimes \Id-\Id \otimes R_{\prec_{A^*}}(x) \big)\beta_\prec(\eta)+\sigma' \Big( \big( L_{\succ_{A^*}} (\eta)\otimes \Id-\Id \otimes \huaR_{A^*}(\eta) \big) \beta_\succ(\xi)\Big)&=&0,
\end{eqnarray}
\end{itemize}
where $\huaL_A= L_{\prec_A}+L_{\succ_A},~\huaR_A= R_{\prec_A}+R_{\succ_A},~\huaL_{A^*}= L_{\prec _{A^*}}+L_{\succ_{A^*}},~\huaR_{A^*}= R_{\prec_{A^*}}+R_{\succ_{A^*}}.$
Moreover, $\sigma:A \otimes A \to A \otimes A $ is the exchange operator defined by
$$  \sigma(x\otimes y)=y\otimes x, \quad \forall x,y \in A ,$$
and $\sigma':A^* \otimes A^* \to A^* \otimes A^*$ is also the exchange operator. We denote a dendriform D-bialgebra by $(A,A^*,\Delta_{\prec},\Delta_{\succ},\beta_{\prec},\beta_{\succ})$, or simply by $(A,A^*).$
\end{defi}

\begin{rmk}
Note that both $(A\otimes A;\Id\otimes L_{\succ_A},\huaR_A\otimes \Id )$ and $(A\otimes A;\Id\otimes \huaL_A,R_{\prec_A}\otimes \Id)$ are the representations of the associative algebra $(A,\ast)$. Equations \eqref{equ1}-\eqref{equ2} imply that $\Delta_{\prec}$ and $\Delta_{\succ}$ are $1$-cocycles of the associative algebra $(A,\ast)$ with coefficients in the representation
$(A\otimes A;\Id\otimes L_{\succ_A},\huaR_A\otimes \Id )$ and $(A\otimes A;\Id\otimes \huaL_A,R_{\prec_A}\otimes \Id)$ respectively. Equations \eqref{equ3}-\eqref{equ4} imply that $\beta_{\prec}$ and $\beta_{\succ}$ are $1$-cocycles of the dual associative algebra $A^\ast$ with coefficients in the representations $(A^\ast\otimes A^\ast;\Id\otimes L_{\succ_{A^\ast}},\huaR_{A^\ast}\otimes \Id )$ and $(A^\ast \otimes A^\ast;\Id\otimes \huaL_{A^\ast},R_{\prec_{A^\ast}}\otimes \Id)$ respectively. Moreover, Equation \eqref{equ5} is a compatibility condition between $\Delta_{\prec}$ and $\Delta_{\succ}$. Equation \eqref{equ6} is a compatibility condition between $\beta_{\prec}$ and $\beta_{\succ}$.
\end{rmk}

\begin{pro}{\rm(\cite{Bai})}
Let $(A,A^*,\Delta_{\prec},\Delta_{\succ},\beta_{\prec},\beta_{\succ})$ be a dendriform D-bialgebra. Then its dual $(A^*,A,$ $\beta_{\prec},\beta_{\succ},\Delta_{\prec},\Delta_{\succ})$ is also a dendriform D-bialgebra.
\end{pro}

\begin{defi}
Suppose that $(A,A^*,\Delta_{\prec},\Delta_{\succ},\beta_{\prec},\beta_{\succ})$ and $(B,B^*,\Delta_{\prec},\Delta_{\succ},\beta_{\prec},\beta_{\succ})$ are two dendriform D-bialgebras. A linear map $\varphi:A\longrightarrow B$ is called a {\bf homomorphism of dendriform D-bialgebras}, if
$\varphi:A\longrightarrow B$ is a homomorphism of dendriform algebras such that $\varphi^*:B^*\longrightarrow A^*$ is also a homomorphism of dendriform algebras, that is, $\varphi$ satisfies
\begin{eqnarray}
&&(\varphi\otimes\varphi)\circ \Delta_{\succ}=\Delta_{\succ}\circ \varphi,\quad (\varphi\otimes\varphi)\circ \Delta_{\prec}=\Delta_{\succ}\circ \varphi,\\
&&(\varphi^*\otimes\varphi^*)\circ \beta_{\succ}=\beta_{\succ}\circ \varphi^*,\quad
(\varphi^*\otimes\varphi^*)\circ \beta_{\prec}=\beta_{\succ}\circ \varphi^*.
\end{eqnarray}
Furthermore, if $\varphi:A\longrightarrow B$ is a linear isomorphism of vector spaces, then $\varphi$ is called an {\bf isomorphism of dendriform D-bialgebras}.
\end{defi}

The notion of a coboundary dendriform D-bialgebra was introduced in \cite{Bai}.

\begin{defi}
A dendriform D-bialgebra $(A,A^*,\Delta_{\prec},\Delta_{\succ},\beta_{\prec},\beta_{\succ})$ is called {\bf coboundary} if both $\Delta_{\prec}$ and $\Delta_{\succ}$ are $1$-coboundaries of the sub-adjacent associative algebra $(A,\ast)$ with coefficients in the representations $(A\otimes A;\Id\otimes \huaL_A,R_{\prec_A}\otimes \Id)$ and $(A\otimes A;\Id\otimes L_{\succ_A},\huaR_A\otimes \Id )$ respectively, that is, there exists $r_\succ,r_\prec\in A\otimes A$ such that
\begin{eqnarray}
\label{del-succ}\Delta_{\succ}(x)&=&(\Id\otimes \huaL_A(x) -R_{\prec_A}(x)\otimes \Id)r_{\succ};\\
\label{del-prec}\Delta_{\prec}(x)&=&(\Id\otimes L_{\succ_A}(x)-\huaR_A(x)\otimes \Id)r_{\prec},\quad \forall x\in A.
\end{eqnarray}
\end{defi}

Let $r=\sum_i a_i\otimes b_i\in A \otimes A$. Let $U(A^c)$ be the universal enveloping algebra of the Lie algebra $A^c$, where $A^c$ is the sub-adjacent Lie algebra of the  sub-adjacent associative algebra $(A,\ast)$ by the commutator.  We introduce $r_{12},r_{13},r_{23}\in U(A^c)$ as follows:
\begin{eqnarray*}
r_{12}=\sum_i a_i\otimes b_i\otimes 1,\quad r_{13}=\sum_i a_i\otimes 1\otimes b_i,\quad r_{23}=\sum_i 1\otimes a_i\otimes b_i.
\end{eqnarray*}
Furthermore, the notations $r_{12}\ast r_{13},r_{13}\prec r_{23}$ and $r_{23}\succ r_{12}$ are defined by
\begin{eqnarray}
\nonumber r_{12}\ast r_{13}&=& \sum_{i,j} (a_i \ast a_j) \otimes b_i \otimes b_j,\\
\label{DYBE}r_{13}\prec r_{23}&=& \sum_{i,j} a_i\otimes a_j \otimes (b_i \prec_A b_j),\\
\nonumber r_{23}\succ r_{12}&=&\sum_{i,j} a_j\otimes  (a_i\succ_A b_j) \otimes b_i.
\end{eqnarray}

Then we consider a special condition:
\begin{eqnarray}\label{r-prec}
 r_\prec=r=\sum_i a_i\otimes b_i,\quad  r_\succ=-\sigma(r)=-\sum_i b_i\otimes a_i.
\end{eqnarray}

\begin{thm}\label{coboundary dendriform}{\rm(\cite{Bai})}
Let $(A,\prec_A,\succ_A)$ be a dendriform algebra and $r=\sum_{i} a_i\otimes b_i \in A\otimes A.$ Then the maps $\Delta_{\succ},\Delta_{\prec}$ defined by \eqref{del-succ} and \eqref{del-prec} with $r_\succ, r_\prec$ satisfying \eqref{r-prec} induce a dendriform algebra structure on $A^*$ such that $(A,A^*)$ is a coboundary dendriform D-bialgebra if and only if $r$ satisfies the following equations
\begin{eqnarray}
\label{condition1}&&\Big(\sigma Q(x) \Big)Q(y)(r-\sigma(r))=0;\\
\label{condition2}&&\Big(R_{\prec_A}(x)\otimes L_{\succ_A}(y)-\Id\otimes L_{\succ_A}(y\prec_A x)-R_{\prec_A}(y\succ_A x)\otimes \Id\Big)(r-\sigma(r))=0;\\
\label{condition3}&&\Big(R_{\prec_A}(x)\otimes \Id \otimes \Id- \Id\otimes \Id\otimes L_{\succ_A}(x)\Big)(-r_{23}\ast r_{21}+r_{21}\prec r_{31}+r_{31}\succ r_{23})=0;\\
\label{condition4}&&\Big(\huaR_A(x)\otimes \Id\otimes \Id-\Id\otimes \Id \otimes L_{\succ_A}(x)\Big)(r_{12}\ast r_{13}-r_{13}\prec r_{23}-r_{23}\succ r_{12})\\
\nonumber&&+\sum_{i} (a_i\ast x)\otimes Q(b_i)(r-\sigma(r))-a_i \otimes \Big( Q(x\succ_A b_i)
(r-\sigma(r)) \Big)=0;\\
\label{condition5}&&\Big(R_{\prec_A}(x)\otimes \Id \otimes \Id- \Id\otimes \Id\otimes \huaL_A(x)\Big)(-r_{31}\ast r_{32}+r_{32}\prec r_{12}+r_{12}\succ r_{31})\\
\nonumber&&+\sum_{i} \Big( Q(b_i)(r-\sigma(r))\otimes (x\ast a_i)-Q(b_i\prec_A x)(r-\sigma(r))\otimes a_i \Big)=0,
\end{eqnarray}
where $Q(x)=\Id\otimes L_{\succ_A}(x)-R_{\prec_A}(x)\otimes \Id$ and $r_{12}\ast r_{13}=r_{12}\prec r_{13}+r_{12}\succ r_{13}$. We denote a coboundary dendriform D-bialgebra by $(A,A_r^*)$.
\end{thm}

\begin{rmk}
In fact, \eqref{condition3}-\eqref{condition5} are satisfied if and only if $(\prec_{A^*}:=\Delta_\prec^*,\succ_{A^*}:=\Delta_\succ^*):A^*\otimes A^* \longrightarrow A^*$ defines a dendriform algebra structure on $A^*$. \eqref{condition1} is equivalent to \eqref{equ6} since $r_\succ, r_\prec$ satisfy \eqref{r-prec}. \eqref{equ1}-\eqref{equ2} are naturally satisfied since $\Delta_{\succ},\Delta_{\prec}$ are $1$-coboundaries defined by \eqref{del-succ} and \eqref{del-prec}. In this case, \eqref{equ5} is also naturally satisfied. Moreover, \eqref{equ3}-\eqref{equ4} are equivalent to \eqref{condition2}.
\end{rmk}

In particular, the equation $r_{12}\ast r_{13}-r_{13}\prec r_{23}-r_{23}\succ r_{12}=0$ is called the {\bf D-equation} on the dendriform algebra $(A,\prec_A,\succ_A)$.

Let $r\in A\otimes A.$ We define $r_{+},r_{-}:A^*\longrightarrow A$ by
\begin{equation}
r(\xi,\eta)=\langle r_{+}(\xi),\eta \rangle=\langle \xi,r_{-}(\eta)\rangle, \quad \forall \xi,\eta\in A^*.
\end{equation}
Then the dendriform algebra structure on $A^*$ defined by \eqref{del-succ} and \eqref{del-prec} is given by
\begin{eqnarray}
\label{succr}\xi \succ_r \eta&=&\huaR_A^* (r_+(\xi)) \eta-L_{\prec_A}^*(r_{-}(\eta))\xi;\\
\label{precr}\xi \prec_r \eta&=&\huaL_A^* (r_+(\eta)) \xi-R_{\succ_A}^*(r_{-}(\xi))\eta, \quad \forall \xi,\eta\in A^*.
\end{eqnarray}
We will denote the dendriform algebra $(A^*,\succ_r,\prec_r)$ by $A^{\ast}_r$.

Now we introduce the notion of $(L_\succ,R_\prec)$-invariance of a $2$-tensor $r\in A\otimes A,$ which is the main ingredient in the definition of a quasi-triangular dendriform D-bialgebra.

\begin{defi}
Let $(A,\prec_A,\succ_A)$ be a dendriform algebra and $r\in A \otimes A$. Then r is called {\bf $(L_\succ,R_\prec)$-invariant} if
\begin{eqnarray}
\big(\Id \otimes \huaL_A(x)-R_{\prec_A}(x)\otimes \Id \big)r&=&0,\\
\big(\Id \otimes L_{\succ_A}(x)-\huaR_A(x)\otimes \Id \big)\sigma(r)&=&0, \quad \forall x\in A.
\end{eqnarray}
\end{defi}

\begin{lem}\label{L-R inv}
Let $(A,\prec_A,\succ_A)$ be a dendriform algebra and $r\in A \otimes A$. Then r is  $(L_\succ,R_\prec)$-invariant if and only if
\begin{eqnarray}
x\ast r_{+}(\xi)&=&r_{+} (R_{\prec_A}^*(x)\xi),\\
r_{+}(\xi)\ast x&=&r_{+} (L_{\succ_A}^*(x)\xi),  \quad \forall x\in A,\xi\in A^*.
\end{eqnarray}
\end{lem}

\begin{proof}
By a direct calculation, for all $x\in A, \xi,\eta\in A^*$, we have
\begin{eqnarray*}
\langle   (\Id \otimes \huaL_A(x)-R_{\prec_A}(x)\otimes \Id)r,\xi\otimes\eta  \rangle
&=&\langle r,\xi\otimes \huaL_A^*(x)\eta \rangle-\langle r, R_{\prec_A}^*(x)\xi\otimes \eta  \rangle\\
&=&\langle  x\ast r_{+}(\xi),\eta\rangle-\langle r_{+}(R_{\prec_A}^*(x)\xi),\eta\rangle\\
&=&\langle x\ast r_{+}(\xi)-r_{+} (R_{\prec_A}^*(x)\xi),\eta  \rangle,
\end{eqnarray*}
and
\begin{eqnarray*}
\langle   (\Id \otimes L_{\succ_A}(x)-\huaR_A(x)\otimes \Id)\sigma(r),\eta\otimes\xi  \rangle
&=&\langle   (L_{\succ_A}(x)\otimes\Id - \Id \otimes \huaR_A(x))r,\xi \otimes \eta \rangle\\
&=&\langle r,L_{\succ_A}^*(x)\xi\otimes\eta \rangle-\langle r,\xi\otimes \huaR_A^*(x) \eta  \rangle\\
&=&\langle r_{+} (L_{\succ_A}^*(x)\xi) ,\eta\rangle-\langle r_{+}(\xi)\ast x,\eta\rangle\\
&=&\langle r_{+} (L_{\succ_A}^*(x)\xi)-r_{+}(\xi)\ast x,\eta  \rangle,
\end{eqnarray*}
which implies that $r$ is $(L_\succ,R_\prec)$-invariant if and only if $x\ast r_{+}(\xi)=r_{+} (R_{\prec_A}^*(x)\xi)$ and $r_{+}(\xi)\ast x=r_{+} (L_{\succ_A}^*(x)\xi)$, for all $x\in A,\xi\in A^*$.
\end{proof}

Denote by $I$ the operator
\begin{eqnarray}\label{I2}
I=r_{+}-r_{-}:A^*\longrightarrow A.
\end{eqnarray}
Note that $I^*=-I$. Actually, $\frac{1}{2} I$ is the contraction of the skew-symmetric part $a$ of $r$, which means that $\frac{1}{2}\langle I(\xi),\eta \rangle=a(\xi,\eta)$. If $r$ is symmetric, then $I=0$.

Now we give another characterization of the $(L_\succ,R_\prec)$-invariant condition.

\begin{lem}\label{inv}
Let $(A,\prec_A,\succ_A)$ be a dendriform algebra and $r\in A \otimes A$. The skew-symmetric part $a$ of $r$ is $(L_\succ,R_\prec)$-invariant if and only if
$$I\circ R_{\prec_A}^*(x)=\huaL_A(x)\circ I,\quad I\circ L_{\succ_A}^*(x)=\huaR_A(x)\circ I,$$
or
$$R_{\prec_A}(x)\circ I= I \circ \huaL_A^*(x),\quad L_{\succ_A}(x) \circ I = I\circ \huaR^*_A(x), $$
where $I$ is given by \eqref{I2}.
\end{lem}

\begin{proof}
By Lemma \ref{L-R inv}, the skew-symmetric part $a$ of $r$ is $(L_\succ,R_\prec)$-invariant if and only if
\begin{equation}\label{skewsymmetric invariance}
x\ast a_{+}(\xi)=a_{+} (R_{\prec_A}^*(x)\xi),\quad a_{+}(\xi)\ast x=a_{+} (L_{\succ_A}^*(x)\xi), \quad \forall x\in A.
\end{equation}
Note that $a_{+}=\frac{1}{2} I.$ Thus the skew-symmetric part $a$ of $r$ is $(L_\succ,R_\prec)$-invariant if and only if
$$I\circ R_{\prec_A}^*(x)=\huaL_A(x)\circ I,\quad I\circ L_{\succ_A}^*(x)=\huaR_A(x)\circ I,$$
which is equivalent to
$$R_{\prec_A}(x)\circ I= I \circ \huaL_A^*(x),\quad  L_{\succ_A}(x) \circ I = I\circ \huaR^*_A(x),$$
by $I^*=-I.$
\end{proof}

\begin{lem}
Let $r=a+\Lambda$ with $a\in \wedge^2 A$ and $\Lambda\in S^2(A).$ If $a$ is $(L_\succ,R_\prec)$-invariant, then
\begin{eqnarray}
\label{lem1}\langle \eta,a_{+}(\xi)\ast x \rangle+\langle \xi, x\succ_A a_{+}(\eta) \rangle&=&0,\\
\label{lem2}\langle \eta, x\ast a_{+}(\xi) \rangle+\langle \xi, a_{+}(\eta)\prec_A x \rangle&=&0,     \quad \forall x\in A, \xi,\eta\in A^*.
\end{eqnarray}
Furthermore, dendriform algebra multiplications $\prec_r, \succ_r$ on $A^*$ given by \eqref{succr}- \eqref{precr} reduce to
\begin{eqnarray}
\label{lem3}\xi \prec_r \eta&=&\huaL_A^* (\Lambda_+(\eta))\xi-R_{\succ_A}^*(\Lambda_+(\xi))\eta,\\
\label{lem4}\xi \succ_r \eta&=&\huaR_A^* (\Lambda_+(\xi))\eta-L_{\prec_A}^*(\Lambda_+(\eta))\xi, \quad \forall \xi,\eta\in A^*.
\end{eqnarray}
\end{lem}

\begin{proof}
Since the skew-symmetric part $a$ of $r$ is $(L_\succ,R_\prec)$-invariant, by Lemma \ref{inv}, we have
\begin{eqnarray*}
\langle \eta,a_{+}(\xi)\ast x \rangle+\langle \xi, x\succ_A a_{+}(\eta) \rangle
&=&\langle \eta, a_{+}(\xi)\ast x  \rangle+\langle L_{\succ_A}^*(x)\xi, a_{+}(\eta) \rangle\\
&=&\langle \eta, a_{+}(\xi)\ast x- a_{+}(L_{\succ_A}^*(x)\xi)\rangle\\
&=&0,
\end{eqnarray*}
and
\begin{eqnarray*}
\langle \eta, x\ast a_{+}(\xi) \rangle+\langle \xi, a_{+}(\eta)\prec_A x \rangle
&=&\langle \eta, x\ast a_{+}(\xi)  \rangle+\langle R_{\prec_A}^*(x)\xi, a_{+}(\eta) \rangle\\
&=&\langle \eta, x\ast a_{+}(\xi)-a_{+}(R_{\prec_A}^*(x)\xi) \rangle\\
&=&0,
\end{eqnarray*}
which implies that \eqref{lem1} and \eqref{lem2} hold.

Moreover, by \eqref{lem1} and \eqref{lem2}, we have
\begin{eqnarray*}
\huaR_A^* (a_+(\xi)) \eta+L_{\prec_A}^*(a_{+}(\eta))\xi &=& \langle \eta, x\ast a_{+}(\xi) \rangle+\langle \xi, a_{+}(\eta)\prec_A x \rangle=0,\\
\huaL_A^* (a_+(\eta)) \xi+R_{\succ_A}^*(a_{+}(\xi))\eta &=& \langle \xi,a_{+}(\eta)\ast x \rangle+\langle \eta, x\succ_A a_{+}(\xi) \rangle=0.
\end{eqnarray*}
Then by \eqref{succr} and \eqref{precr}, we have
\begin{eqnarray*}
\xi \prec_r \eta&=& \huaL_A^*(\Lambda_+(\eta))\xi+\huaL_A^*(a_+(\eta))\xi-R_{\succ_A}^*(\Lambda_+(\xi))\eta+R_{\succ_A}^*(a_+(\xi))\eta\\
&=&\huaL_A^* (\Lambda_+(\eta))\xi-R_{\succ_A}^*(\Lambda_+(\xi))\eta,\\
\xi \succ_r \eta&=& \huaR_A^* (\Lambda_+(\xi))\eta+\huaR_A^* (a_+(\xi))\eta-L_{\prec_A}^*(\Lambda_+(\eta))\xi+L_{\prec_A}^*(a_+(\eta))\xi\\
&=&\huaR_A^* (\Lambda_+(\xi))\eta-L_{\prec_A}^*(\Lambda_+(\eta))\xi,
\end{eqnarray*}
which implies that \eqref{lem3} and \eqref{lem4} hold.
\end{proof}

\begin{rmk}
By \eqref{lem1} and \eqref{lem2}, we also have
\begin{eqnarray}\label{lem5}
\langle  \xi,a_+(\eta)\succ_A x \rangle=\langle \eta,x\prec_A a_+(\xi)\rangle,\quad \forall x\in A,\xi,\eta\in A^*.
\end{eqnarray}
\end{rmk}

\begin{pro}\label{quasi-triangular}
Let $(A,\prec_A,\succ_A)$ be a dendriform algebra and $r=\sum_{i} a_i\otimes b_i \in A\otimes A.$ If $r$ satisfies the D-equation and the skew-symmetric part $a$ of $r$ is $(L_\succ,R_\prec)$-invariant, then $r$ satisfies \eqref{condition1}-\eqref{condition5}. Therefore, $(A,A_r^*)$ is a coboundary dendriform D-bialgebra.
\end{pro}

\begin{proof}
Now we prove that \eqref{condition1}-\eqref{condition5} are satisfied.
Since $r\in A \otimes A$ satisfies the D-equation and the skew-symmetric part $a$ of $r$ is $(L_\succ,R_\prec)$-invariant, by \eqref{skewsymmetric invariance}, we have
\begin{eqnarray*}
&&  \langle (\sigma Q(x))Q(y)a, \xi\otimes \eta  \rangle\\
&=& \langle (L_{\succ_A}(x)\otimes 1-1\otimes R_{\prec_A}(x))(1\otimes L_{\succ_A}(y)-R_{\prec_A}(y)\otimes 1)a, \xi\otimes \eta \rangle\\
&=& \langle \Big( L_{\succ_A}(x)\otimes L_{\succ_A}(y)-1\otimes R_{\prec_A}(x)L_{\succ_A}(y)-L_{\succ_A}(x)R_{\prec_A}(y)\otimes 1+ R_{\prec_A}(y)\otimes R_{\prec_A}(x)\Big)a, \xi\otimes \eta \rangle  \\
&=& \langle a, L_{\succ_A}^*(x)\xi \otimes L_{\succ_A}^*(y)\eta \rangle- \langle a, \xi \otimes L_{\succ_A}^*(y)R_{\prec_A}^*(x)\eta \rangle-\langle a,R_{\prec_A}^*(y)L_{\succ_A}^*(x)\xi \otimes \eta \rangle+\langle a, R_{\prec_A}^*(y)\xi \otimes R_{\prec_A}^*(x) \eta \rangle\\
&=& \langle y\succ_A a_+(L_{\succ_A}^*(x)\xi), \eta \rangle-\langle R_{\prec_A}(x)L_{\succ_A}(y)(a_+(\xi)) ,\eta \rangle-\langle a_+(R_{\prec_A}^*(y)L_{\succ_A}^*(x)\xi),\eta \rangle+\langle a_+(R_{\prec_A}^*(y)\xi), R_{\prec_A}^*(x)\eta  \rangle \\
&=& \langle y\succ_A(a_+(\xi)\ast x),\eta \rangle-\langle (y\succ_A a_+(\xi))\prec x, \eta \rangle-\langle y \ast (a_+(\xi)\ast x),\eta \rangle+\langle (y\ast a_+(\xi))\prec_A x, \eta \rangle\\
&=& -\langle y \prec_A (a_+(\xi)\ast x),\eta \rangle+\langle (y\prec_A a_+(\xi))\prec_A x, \eta \rangle\\
&=& 0,
\end{eqnarray*}
which implies that \eqref{condition1} holds. By \eqref{skewsymmetric invariance}, we also have
\begin{eqnarray*}
&& \langle\Big(R_{\prec_A}(x)\otimes L_{\succ_A}(y)-\Id\otimes L_{\succ_A}(y\prec_A x)-R_{\prec_A}(y\succ_A x)\otimes \Id\Big)a, \xi\otimes \eta  \rangle\\
&=&  \langle  a, R_{\prec_A}^*(x)\xi \otimes L_{\succ_A}^*(y)\eta - \xi\otimes L_{\succ_A}^*(y\prec_A x)\eta  \rangle - R_{\prec_A}^*(y\succ_A x) \xi\otimes \eta \rangle\\
&=&  \langle  a_+(R_{\prec_A}^*(x)\xi), L_{\succ_A}^*(y)\eta \rangle- \langle a_+(\xi),L_{\succ_A}^*(y\prec_A x)\eta  \rangle-\langle a_+(R_{\prec_A}^*(y\succ_A x) \xi),\eta  \rangle  \\
&=&  \langle y\succ_A( x\ast  a_+(\xi)),\eta \rangle-\langle (y\prec_A x)\succ_A a_+(\xi),\eta \rangle-\langle (y\prec_A x)\ast a_+(\xi),\eta \rangle\\
&=&  \langle y\succ_A ( x \succ_A  a_+(\xi)),\eta \rangle+\langle y\succ_A( x\prec_A  a_+(\xi)),\eta \rangle-\langle (y\prec_A x)\succ_A  a_+(\xi),\eta \rangle\\
&&-\langle  (y\succ_A x)\succ_A  a_+(\xi),\eta \rangle -\langle (y\succ_A x)\prec_A  a_+(\xi),\eta \rangle \\
&=&  0,
\end{eqnarray*}
which implies that \eqref{condition2} holds.
Let $ \sigma_{123},\sigma_{132}:A\otimes A\otimes A\longrightarrow A\otimes A\otimes A$ be two maps given by
$$ \sigma_{123}(x\otimes y \otimes z)= z\otimes x\otimes y, \quad \sigma_{132}(x\otimes y \otimes z)= y\otimes z\otimes x,\quad \forall x,y,z\in A.    $$
Then we have
\begin{eqnarray*}
r_{23}\ast r_{21}-r_{21}\prec r_{31}-r_{31}\succ r_{23}&=&\sigma_{123}(r_{12}\ast r_{13}-r_{13}\prec r_{23}-r_{23}\succ r_{12});\\
r_{31}\ast r_{32}-r_{32}\prec r_{12}-r_{12}\succ r_{31}&=&\sigma_{132}(r_{12}\ast r_{13}-r_{13}\prec r_{23}-r_{23}\succ r_{12}).
\end{eqnarray*}
Therefore, if $r$ satisfies the D-equation, then \eqref{condition3} holds. Moreover, we also have
\begin{eqnarray*}
&&\sum_{i} \langle (a_i\ast x)\otimes Q(b_i)a-a_i \otimes Q(x\succ_A b_i)a, \xi \otimes \eta \otimes \zeta \rangle \\
&=&\sum_{i} \Big( \langle a_i\ast x,\xi\rangle \langle a, \eta\otimes L_{\succ_A}^*(b_i) \zeta-R_{\prec_A}^*(b_i)\eta\otimes \zeta \rangle- \langle  a_i,\xi \rangle \langle  a,\eta\otimes L_{\succ_A}^*(x\succ_A b_i) \zeta-R_{\prec_A}^*(x\succ_A b_i) \eta \otimes \zeta \rangle \Big)\\
&=&\sum_{i} \Big( \langle a_i\ast x,\xi\rangle \langle b_i\succ_A a_+(\eta)- b_i\ast a_+(\eta),\zeta \rangle-\langle  a_i,\xi \rangle \langle  (x\succ_A b_i)\succ_A a_+(\eta)-(x\succ_A b_i)\ast a_+(\eta),\zeta\rangle \Big) \\
&=&\sum_{i} \Big( \langle  a_i,\xi \rangle \langle (x\succ_A b_i)\prec_A a_+(\eta),\zeta \rangle-\langle a_i\ast x,\xi\rangle \langle b_i\prec_A a_+(\eta),\zeta \rangle \Big)\\
&=&\langle (x\succ_A r_+(\xi))\prec_A a_+(\eta),\zeta \rangle-\langle (x\succ_A r_+(\xi))\prec_A a_+(\eta),\zeta \rangle\\
&=& 0,
\end{eqnarray*}
which implies \eqref{condition4} holds.
Similarly, by \eqref{skewsymmetric invariance}, we have
\begin{eqnarray*}
&&\sum_{i} \langle Q(b_i)a \otimes (x\ast a_i)-Q(b_i\prec_A x)a \otimes a_i , \xi\otimes \eta\otimes \zeta \rangle\\
&=&\sum_{i} \Big( \langle b_i\succ_A a_+(\xi)-a_+(R_{\prec_A}^*(b_i)\xi),\eta \rangle \langle x\ast a_i,\zeta \rangle -\langle (b_i\prec_A x)\succ_A a_+(\xi)-a_+(R_{\prec_A}^*(b_i\prec_A x)\xi),\eta \rangle \langle a_i,\zeta  \rangle \Big) \\
&=&\sum_{i} \Big( \langle b_i\succ_A a_+(\xi)-b_i\ast a_+(\xi),\eta \rangle \langle x\ast a_i,\zeta \rangle -\langle (b_i\prec_A x)\succ_A a_+(\xi)-(b_i\prec_A x)\ast a_+(\xi),\eta \rangle \langle a_i,\zeta  \rangle \Big)\\
&=&\sum_{i} \Big( \langle (b_i\prec_A x)\prec_A a_+(\xi),\eta \rangle \langle a_i,\zeta  \rangle-\langle b_i\prec_A a_+(\xi),\eta \rangle \langle x\ast a_i,\zeta \rangle \Big) \\
&=&\langle (r_+(\zeta)\prec_A x)\prec_A a_+(\xi),\eta  \rangle-\langle (r_+(\zeta)\prec_A x)\prec_A a_+(\xi),\eta  \rangle\\
&=& 0,
\end{eqnarray*}
which implies \eqref{condition5} holds. Therefore, $(A,A_r^*)$ is a coboundary dendriform D-bialgebra by Theorem \ref{coboundary dendriform}.
\end{proof}

Based on the observation in Proposition \ref{quasi-triangular}, we introduce the notion of quasi-triangular dendriform D-bialgebras as a special class of coboundary dendriform D-bialgebras as follows.

\begin{defi}
Let $(A,\prec_A,\succ_A)$ be a dendriform algebra. If $r\in A \otimes A$ satisfies the D-equation and the skew-symmetric part $a$ of $r$ is $(L_\succ,R_\prec)$-invariant, then the dendriform D-bialgebra $(A,A_r^*)$ induced by $r$ is called a
{\bf quasi-triangular dendriform D-bialgebra}. Moreover, if $r$ is symmetric, $(A,A_r^*)$ is called a
{\bf triangular dendriform D-bialgebra}.
\end{defi}

\begin{thm}\label{heart2}
Assume that the skew-symmetric part $a$ of $r\in A\otimes A$ is $(L_\succ,R_\prec)$-invariant. Then
$r$ satisfies the D-equation $r_{12}\ast r_{13}-r_{13}\prec r_{23}-r_{23}\succ r_{12}=0$ if and only if
$(A^*,\prec_r,\succ_r)$ is a dendriform algebra and the linear maps $r_{+},r_{-}:(A^*,\prec_r,\succ_r)\longrightarrow (A,\prec_A,\succ_A)$ are both dendriform algebra homomorphisms.
\end{thm}

\begin{proof}
By Theorem \ref{coboundary dendriform} and Proposition \ref{quasi-triangular}, if the skew-symmetric part $a$ of $r$ is $(L_\succ,R_\prec)$-invariant and $r$ satisfies the D-equation, $(A,A^*_r)$ induced by $r$ is a quasi-triangular dendriform D-bialgebra, which implies that $(A^*,\prec_r,\succ_r)$ is a dendriform algebra. Next we show that the linear maps $r_{+},r_{-}$ are both dendriform algebra homomorphisms.

By \eqref{DYBE} and \eqref{lem1}-\eqref{lem2}, after a direct calculation, we have
\begin{eqnarray*}
&&\langle r_{12}\ast r_{13}-r_{13}\prec r_{23}-r_{23}\succ r_{12},\xi\otimes \eta\otimes \zeta\rangle\\
&=&\langle \xi,\Lambda_{+}(\eta)\ast\Lambda_{+}(\zeta) \rangle-\langle \zeta,\Lambda_{+}(\xi)\prec_A \Lambda_{+}(\eta) \rangle-\langle \eta,\Lambda_{+}(\zeta)\succ_A\Lambda_{+}(\xi) \rangle-\langle\zeta,a_{+}(\xi)\prec_A a_{+}(\eta)\rangle.
\end{eqnarray*}
Next, we show that the following equations hold.
\begin{eqnarray}
\label{equ-succ}\langle r_+(\xi \succ_r \eta)-r_+(\xi)\succ_A r_+(\eta),\zeta\rangle&=&\langle r_{12}\ast r_{13}-r_{13}\prec r_{23}-r_{23}\succ r_{12},\eta\otimes \zeta \otimes \xi \rangle;\\
\label{equ-prec}\langle r_+(\xi \prec_r \eta)-r_+(\xi)\prec_A r_+(\eta),\zeta\rangle&=&\langle r_{12}\ast r_{13}-r_{13}\prec r_{23}-r_{23}\succ r_{12},\xi \otimes \eta \otimes \zeta \rangle.
\end{eqnarray}
On the one hand, by \eqref{lem4} and \eqref{lem5}, we have
\begin{eqnarray*}
&&\langle r_+(\xi \succ_r \eta)-r_+(\xi)\succ_A r_+(\eta),\zeta\rangle\\
&=& \langle (a_+ +\Lambda_+)(\huaR_A^* (\Lambda_+(\xi))\eta-L_{\prec_A}^*(\Lambda_+(\eta))\xi),\zeta\rangle
-\langle (a_+ +\Lambda_+)(\xi)\succ_A (a_+ +\Lambda_+)(\eta),\zeta\rangle\\
&=& -\langle \eta, a_+(\zeta)\ast \Lambda_+(\xi) \rangle+\langle \eta, \Lambda_+(\zeta)\ast \Lambda_+(\xi) \rangle+\langle \xi, \Lambda_+(\eta)\prec_A a_+(\zeta) \rangle-\langle \xi,\Lambda_+(\eta)\prec_A \Lambda_+(\zeta) \rangle\\
&&-\langle \zeta,a_+(\xi)\succ_A a_+(\eta)   \rangle-\langle \zeta,a_+(\xi) \succ_A \Lambda_+(\eta) \rangle-\langle \zeta, \Lambda_+(\xi) \succ_A a_+(\eta) \rangle-\langle \zeta, \Lambda_+(\xi) \succ_A \Lambda_+(\eta)\rangle\\
&=& \langle \eta, \Lambda_+(\zeta)\ast \Lambda_+(\xi) \rangle-\langle \xi,\Lambda_+(\eta)\prec_A \Lambda_+(\zeta) \rangle-\langle \zeta,a_+(\xi)\succ_A a_+(\eta)   \rangle-\langle \zeta, \Lambda_+(\xi) \succ_A \Lambda_+(\eta)\rangle\\
&=& \langle \eta, \Lambda_+(\zeta)\ast \Lambda_+(\xi) \rangle-\langle \xi,\Lambda_+(\eta)\prec_A \Lambda_+(\zeta) \rangle-\langle \xi,a_+(\eta)\prec_A a_+(\zeta) \rangle-\langle \zeta, \Lambda_+(\xi) \succ_A \Lambda_+(\eta)\rangle,
\end{eqnarray*}
which implies that \eqref{equ-succ} holds.

On the other hand, by \eqref{lem3} and \eqref{lem5}, we have
\begin{eqnarray*}
&&\langle r_+(\xi \prec_r \eta)-r_+(\xi)\prec_A r_+(\eta),\zeta\rangle\\
&=& \langle (a_+ +\Lambda_+)(\huaL_A^* (\Lambda_+(\eta))\xi-R_{\succ_A}^*(\Lambda_+(\xi))\eta),\zeta\rangle
-\langle (a_+ +\Lambda_+)(\xi)\prec_A (a_+ +\Lambda_+)(\eta),\zeta\rangle\\
&=&-\langle \xi,\Lambda_+(\eta)\ast a_+(\zeta) \rangle+\langle \xi,\Lambda_+(\eta)\ast \Lambda_+(\zeta) \rangle+\langle \eta,a_+(\zeta)\succ_A \Lambda_+(\xi)\rangle-\langle \eta,\Lambda_+(\zeta)\succ_A \Lambda_+(\xi)  \rangle\\
&&-\langle \zeta,a_+(\xi)\prec_A a_+(\eta) \rangle-\langle \zeta,a_+(\xi)\prec_A \Lambda_+(\eta) \rangle-\langle  \zeta,\Lambda_+(\xi)\prec_A a_+(\eta) \rangle-\langle \zeta,\Lambda_+(\xi)\prec_A \Lambda_+(\eta) \rangle\\
&=& \langle \xi,\Lambda_+(\eta)\ast \Lambda_+(\zeta) \rangle-\langle \eta,\Lambda_+(\zeta)\succ_A \Lambda_+(\xi)  \rangle-\langle \zeta,a_+(\xi)\prec_A a_+(\eta) \rangle-\langle \zeta,\Lambda_+(\xi)\prec_A \Lambda_+(\eta) \rangle,
\end{eqnarray*}
which implies that \eqref{equ-prec} holds. Thus $r_+$ is a dendriform algebra homomorphism. Similarly, we can also prove that $r_-$ is a dendriform algebra homomorphism.

Conversely, if $r_+$ is a dendriform algebra homomorphism, by \eqref{equ-succ} and \eqref{equ-prec}, then $r$ satisfies the D-equation $r_{12}\ast r_{13}-r_{13}\prec r_{23}-r_{23}\succ r_{12}=0$.
\end{proof}

\subsection{Factorizable dendriform D-bialgebras}

In this subsection, we introduce the notion of a factorizable dendriform D-bialgebra, which is a special quasi-triangular dendriform D-bialgebra.
Factorizable dendriform D-bialgebras are the case when the map $I:
A^*\to A$ is nondegenerate.

\begin{defi}
A quasi-triangular dendriform D-bialgebra $(A,A_{r}^{\ast})$ is called {\bf factorizable} if the skew-symmetric part $a$ of $r$ is nondegenerate, which means that the linear map $I:A^{\ast} \longrightarrow A$ defined in \eqref{I2} is a linear isomorphism of vector spaces.
\end{defi}

Consider the map
$$
A^*\stackrel{r_+\oplus r_-}{\longrightarrow}A\oplus A\stackrel{(x,y)\longmapsto x-y}{\longrightarrow}A.
$$
The following result justifies the terminology of a factorizable dendriform D-bialgebra.
\begin{pro}
  Let  $(A,A_{r}^{\ast})$ be a factorizable dendriform D-bialgebra. Then $\Img(r_+\oplus r_-)$ is a dendriform subalgebra of the direct sum dendriform algebra $A\oplus A$, which is isomorphic to the dendriform algebra $(A^*,\prec_r,\succ_r).$ Moreover, any $x\in A$ has a unique decomposition
  \begin{equation}
    x=x_+-x_-,
  \end{equation}
  where $(x_+,x_-)\in \Img(r_+\oplus r_-)$.
\end{pro}
\begin{proof}
By Theorem \ref{heart2}, both $r_+$ and $r_-$ are dendriform algebra homomorphisms from the dendriform algebra $(A^*,\prec_r,\succ_r)$ to $(A,\prec_A,\succ_A)$. Therefore, $\Img(r_+\oplus r_-)$ is a dendriform subalgebra of the direct sum dendriform algebra $A \oplus A$. Since $I=r_+  -r_-:A^*\to A$ is nondegenerate, it follows that the dendriform algebra $\Img(r_+\oplus r_-)$ is   isomorphic to the dendriform algebra $(A^*,\prec_r,\succ_r).$

  Since $I:A^*\to A$ is nondegenerate, we have
  $$
 r_+ I^{-1}(x)-r_-I^{-1}(x)= (r_+  -r_-)I^{-1}(x)=x,
  $$
  which implies that $x=x_+-x_-,$ where $x_+=r_+ I^{-1}(x)$ and $x_-=r_-I^{-1}(x)$. The uniqueness also follows from the fact that $I:A^*\to A$ is nondegenerate.
\end{proof}

Let $(A,A^*)$ be an arbitrary dendriform D-bialgebra. Define multiplications $\prec_{\frkd}$ and $\succ_{\frkd}$ on $\frkd=A\oplus A^*$ by
\begin{eqnarray*}
(x,\xi)\succ_{\frkd}(y,\eta)
&=& \Big(x\succ_A y-L_{\prec_{A^*}}^*(\eta)x+(R_{\prec_{A^*}}^*+R_{\succ_{A^*}}^*)(\xi)y,\xi \succ_{A^*} \eta+(R_{\prec_A}^*+R_{\succ_A}^*)(x)\eta-L_{\prec_A}^*(y)\xi \Big);\\
(x,\xi)\prec_{\frkd}(y,\eta)
&=& \Big(x\prec_A y+(L_{\prec_{A^*}}^*+L_{\succ_{A^*}}^*)(\eta)x-R_{\succ_{A^*}}^*(\xi)y,\xi \prec_{A^*} \eta-R_{\succ_A}^*(x)\eta+(L_{\prec_A}^*+L_{\succ_A}^*)(y)\xi \Big),
\end{eqnarray*}
for all $x,y\in A, \xi,\eta\in A^*.$ Then $(\frkd,\prec_{\frkd},\succ_{\frkd})$ is a dendriform algebra \cite{Bai}, which is called the {\bf dendriform double} of the dendriform D-bialgebra $(A,A^*)$.

Let $\{ e_1,e_2,...,e_n \}$ be a basis of $A$ and $\{ e_1^{\ast},e_2^{\ast},...,e_n^{\ast} \}$ be the dual basis of $A^{\ast}$. Then $r=\sum_{i} e_i\otimes e_i^{\ast}\in A\otimes A^{\ast} \subset \frkd \otimes \frkd$  induces a dendriform algebra structure on $\frkd^{\ast}$ such that $(\frkd,\frkd^{\ast}_r)$ is a coboundary dendriform D-bialgebra \cite{Bai}, where the dendriform multiplications $\prec_{\frkd^{\ast}_r},\succ_{\frkd^{\ast}_r}$ on $\frkd^{\ast}_r$ are given by
\begin{eqnarray}
\label{frkdast1}(\xi,x)\prec_{\frkd^{\ast}_r}(\eta,y)&=&(\xi\prec_{A^{\ast}}\eta,x\prec_{A} y),\\
\label{frkdast2}(\xi,x)\succ_{\frkd^{\ast}_r}(\eta,y)&=&(\xi\succ_{A^{\ast}}\eta,x\succ_{A} y), \quad \forall x,y\in A, \xi,\eta\in A^{\ast}.
\end{eqnarray}

\begin{thm}\label{dendriform double}
The dendriform D-bialgebra $(\frkd,\frkd^{\ast}_r)$ is a quasi-triangular dendriform D-bialgebra. Moreover, it is factorizable.
\end{thm}

\begin{proof}
We first prove that the skew-symmetric part $a=\frac{1}{2} \sum_{i} (e_i\otimes e_i^{\ast}-e_i^{\ast}\otimes e_i) $ of $r$ is $(\widetilde{L_{\succ}},\widetilde{R_{\prec}})$-invariant, where $\widetilde{L_{\succ}},\widetilde{R_{\prec}}$ are the left and right multiplication operators of the dendriform algebra $(\frkd,\cdot_\frkd)$ respectively.  For all  $(\xi, x)\in \frkd^{\ast}$, we have $a_{+}(\xi, x)=\frac{1}{2} (-x,\xi)\in \frkd.$ Furthermore, by a direct calculation,  we have
\begin{eqnarray*}
(x,\xi)\ast_{\frkd}a_{+}(\eta,y)
&=& (x,\xi)\prec_{\frkd}a_{+}(\eta,y)+(x,\xi)\succ_{\frkd}a_{+}(\eta,y)\\
&=& \frac{1}{2} \Big((x,\xi)\prec_{\frkd}(-y,\eta)+(x,\xi)\succ_{\frkd}(-y,\eta) \Big)\\
&=& \frac{1}{2}\Big(-x\ast_A y+L^*_{\succ_{A^*}}(\eta)x-R^*_{\prec_{A^*}}(\xi)y, \xi \ast_{A^*} \eta +R^*_{\prec_A} (x)\eta-L^*_{\succ_A}(y)\xi \Big),
\end{eqnarray*}
and
\begin{eqnarray*}
R^*_{\prec_\frkd} (x,\xi)(\eta,y)=(\xi \ast_{A^*} \eta +R^*_{\prec_A}(x)\eta-L^*_{\succ_A}(y)\xi,x\ast_{A} y-L^*_{\succ_{A^*}}(\eta)x+R^*_{\prec_{A^*}}(\xi)y).
\end{eqnarray*}
Thus we have
$$a_+ \Big(R^*_{\prec_\frkd} (x,\xi)(\eta,y)\Big)-(x,\xi)\ast_{\frkd}a_{+}(\eta,y)=0.$$
Similarly, we also have
$$a_+ \Big(L^*_{\succ_\frkd} (x,\xi)(\eta,y)\Big)-a_{+}(\eta,y)\ast_{\frkd}(x,\xi)=0.$$
By Lemma \ref{inv},  the skew-symmetric part $a$ of $r$ is  $(\widetilde{L_{\succ}},\widetilde{R_{\prec}})$-invariant. Thus the dendriform D-bialgebra $(\frkd,\frkd^{\ast}_r)$ is a quasi-triangular dendriform D-bialgebra.

Moreover, note that $r_{+},r_{-}:\frkd^{\ast} \longrightarrow \frkd$ are given by
$$ r_{+}(\xi,x)=(0,\xi),\quad r_{-}(\xi,x)=(x,0),\quad \forall x\in A, \xi\in A^{\ast}.$$
This implies that $I(\xi,x)=(-x,\xi)$, which means that the linear map $I:\frkd^{\ast} \longrightarrow \frkd$ is a linear isomorphism of vector spaces. Thus, $(\frkd,\frkd^{\ast}_r)$ is a factorizable dendriform D-bialgebra.
\end{proof}

\section{Quasi-triangular dendriform D-bialgebras and relative Rota-Baxter operators of nonzero weights    }\label{nonzero-rRBo}

\begin{defi}\label{action}
Let $(A,\prec_A,\succ_A)$ and $(B,\prec_B,\succ_B)$ be two dendriform algebras. Then $(l_{\succ},r_{\succ},l_{\prec},r_{\prec})$ is called {\bf an action} of $A$ on $B$, if $(B;l_{\succ},r_{\succ},l_{\prec},r_{\prec})$ is a representation of the dendriform algebra $(A,\prec_A,\succ_A)$ and four linear maps $l_{\succ},r_{\succ},l_{\prec},r_{\prec}$ satisfy the following equations additionally:
\begin{equation*}
\left\{ \begin{aligned}
r_\prec(x)(u \succ_B v)&=u \succ_B (r_\prec(x)v);&~~ (l_\succ(x)u)\prec_B v&=l_\succ(x)(u \prec_B v);&~~ (r_\succ(x)u)\prec_B v&=u\succ_B (l_\prec(x) v);\\
r_\prec(x)(u \prec_B v)&=u \prec_B (r_\ast(x)v);&~~ (l_\prec(x)u)\prec_B v&=l_\prec(x)(u \ast_B v);&~~ (r_\prec(x)u)\prec_B v&=u\prec_B (l_\ast(x) v);\\
l_\succ(x)(u\succ_B v)&=(l_\ast(x)u)\succ_B v;&~~ u\succ_B (l_\succ(x)v)&=(r_\ast(x)u)\succ_B v;&~~ u\succ_B (r_\succ(x)v)&=r_\succ(x)(u \ast_B v),
\end{aligned}\right.
\end{equation*}
for all $x\in A, u,v \in B,$ where $l_{\ast}=l_{\prec}+l_{\succ},~ r_{\ast}=r_{\prec}+r_{\succ}$ and $u \ast_B v=u\prec_B v+u\succ_B v.$
\end{defi}

The notion of Rota-Baxter operators on associative algebras was introduced in \cite{Baxter,Rota} and have various applications \cite{BBGN,CK,GK,Yu-Guo}. See also the book \cite{Guo} for more details. Next we recall the notion of (relative) Rota-Baxter operators of nonzero weights on dendriform algebras.

\begin{defi}{\rm(\cite{BGM})}\label{rRBo}
\begin{itemize}
\item[{\rm(i)}] A linear map $P:A \rightarrow A$ is called a {\bf Rota-Baxter operator of weight $\lambda$} on a dendriform algebra  $(A,\prec_A,\succ_A)$ if
\begin{eqnarray*}
P(x)\prec_A P(y)&=&P(P(x)\prec_A y+x \prec_A P(y)+\lambda x \prec_A y);\\
P(x)\succ_A P(y)&=&P(P(x)\succ_A y+x \succ_A P(y)+\lambda x \succ_A y),\quad  \forall x,y\in A.
\end{eqnarray*}
A {\bf Rota-Baxter dendriform algebra of weight $\lambda$} is a dendriform algebra equipped with a Rota-Baxter operator of weight $\lambda$.
\item[{\rm(ii)}] Let $(A,\prec_A,\succ_A),(B,\prec_B,\succ_B)$ be two dendriform algebras and $(l_{\succ},r_{\succ},l_{\prec},r_{\prec})$ is an action of $A$ on $B$. Then a linear map $T:B\to A$ is called a {\bf relative Rota-Baxter operator of weight $\lambda$} with respect to the action $(l_{\succ},r_{\succ},l_{\prec},r_{\prec})$ of $A$ on $B$, if $T$ satisfies the following equations:
\begin{eqnarray}
\label{rRB1} (Tx)\succ_A(Ty)&=&T(l_\succ(Tx)y+r_\succ(Ty)x+\lambda x\succ_B y);\\
\label{rRB2} (Tx)\prec_A(Ty)&=&T(l_\prec(Tx)y+r_\prec(Ty)x+\lambda x\prec_B y),\quad \forall x,y\in B.
\end{eqnarray}
\end{itemize}
\end{defi}
Note that a Rota-Baxter operator on a dendriform algebra is a relative Rota-Baxter operator with respect to the regular representation.

By $I=r_{+}- r_{-}$, dendriform algebra multiplications $\prec_r, \succ_r$ on $A^*$ given by \eqref{succr} and \eqref{precr} reduce to
\begin{eqnarray}
\label{succr2}\xi \succ_r \eta
&=&\huaR_A^* (r_+(\xi))\eta-L_{\prec_A}^*(r_{+}(\eta))\xi+L_{\prec_A}^*(I\eta)\xi;\\
\label{precr2}\xi \prec_r \eta
&=&\huaL_A^* (r_+(\eta)) \xi-R_{\succ_A}^*(r_{+}(\xi))\eta+R_{\succ_A}^*(I\xi)\eta,
\quad \forall \xi,\eta\in A^*.
\end{eqnarray}

Now we define new multiplications $\prec_+,\succ_+$ on $A^{\ast}$ as follows:
\begin{eqnarray}
\label{newproduct1}\xi \succ_+ \eta&=&L_{\prec_A}^*(I\eta)\xi;\\
\label{newproduct2}\xi \prec_+ \eta&=&R_{\succ_A}^*(I\xi)\eta, \quad \forall \xi,\eta\in A^{\ast}.
\end{eqnarray}

\begin{pro}\label{new-dendriform}
$(A^{\ast},\prec_+,\succ_+)$ is a dendriform algebra.
\end{pro}

\begin{proof}
By Lemma \ref{inv}, for all $x\in A,~~ \xi,\eta,\theta\in A^\ast$, we have
\begin{eqnarray*}
\langle (\xi \succ_+ \eta)\prec_+ \theta-\xi \succ_+ (\eta \prec_+ \theta), x \rangle
&=& \langle R_{\succ_A}^*(I(L_{\prec_A}^*(I\eta)\xi))\theta- L_{\prec_A}^*(I(R_{\succ_A}^*(I\eta)\theta))\xi,x  \rangle\\
&=& \langle L_{\succ_A}^*(x)\theta,I(L_{\prec_A}^*(I\eta)\xi)\rangle-\langle \xi,I(R_{\succ_A}^*(I\eta)\theta)\prec_A x \rangle\\
&=& -\langle I(L_{\succ_A}^*(x)\theta),L_{\prec_A}^*(I\eta)\xi\rangle+\langle \xi,L_{\prec_A}(I\eta)(I\theta)\prec_A x\rangle\\
&=& -\langle L_{\prec_A}(I\eta)(\huaR_A(x)I\theta),\xi \rangle+\langle \xi,(I\eta\prec_A I\theta) \prec_A x \rangle\\
&=& \langle \xi,-I\eta\prec_A(I\theta \ast x)+(I\eta\prec_A I\theta) \prec_A x \rangle\\
&=& 0,
\end{eqnarray*}
which implies that $(\xi \succ_+ \eta)\prec_+ \theta=\xi \succ_+ (\eta \prec_+ \theta)$. Moreover, we have
\begin{eqnarray*}
&&\langle (\xi \prec_+ \eta)\prec_+ \theta-\xi \prec_+ (\eta \prec_+ \theta+\eta \succ_+ \theta), x \rangle\\
&=& \langle R^*_{\succ_A}(I(R^*_{\succ_A}(I\xi)\eta))\theta-R^*_{\succ_A}(I\xi)R^*_{\succ_ A}(I\eta)\theta-R^*_{\succ_A}(I\xi)L^*_{\prec_A}(I\theta)\eta,x  \rangle\\
&=& \langle \theta, x\succ_A I(R^*_{\succ_A}(I\xi)\eta)-(x \succ_A I\xi)\succ_A I\eta \rangle-
\langle R^*_{\prec_A}(x\succ_A I\xi)\eta,I\theta  \rangle\\
&=& \langle \theta,-x\succ_A (I\xi \prec_A I\eta)-(x \succ_A I\xi)\succ_A I\eta+(x\succ_A I\xi)\ast I\eta   \rangle\\
&=& \langle \theta,-x\succ_A (I\xi \prec_A I\eta)+(x\succ_A I\xi)\prec_A I\eta   \rangle\\
&=& 0,
\end{eqnarray*}
which implies that $(\xi \prec_+ \eta)\prec_+ \theta=\xi \prec_+ (\eta \prec_+ \theta+\eta \succ_+ \theta)$. Similarly, we also have
\begin{eqnarray*}
&&\langle \xi \succ_+ (\eta \succ_+ \theta)-(\xi \prec_+\eta+\xi \succ_+\eta)\succ_+ \theta, x \rangle\\
&=& \langle \eta,-I\theta \prec_A (x\ast I\xi)-(I\theta \prec_A x)\succ_A I\xi+(I\theta \prec_A x)\ast I\xi \rangle\\
&=& \langle \eta, -I\theta \prec_A (x\ast I\xi)+(I\theta \prec_A x)\prec_A I\xi \rangle\\
&=& 0,
\end{eqnarray*}
which implies that $\xi \succ_+ (\eta \succ_+ \theta)=(\xi \prec_+\eta+\xi \succ_+\eta)\succ_+ \theta$.
Thus, $(A^{\ast},\prec_+,\succ_+)$ is a dendriform algebra.
\end{proof}

\begin{lem}\label{new-action}
$(\huaR_A^\ast,-L_{\prec_A}^\ast,-R_{\succ_A}^\ast,\huaL_A^\ast)$ is an action of the dendriform algebra $(A,\prec_A,\succ_A)$ on the dendriform algebra $(A^\ast,\prec_+,\succ_+)$ given in Proposition \ref{new-dendriform}.
\end{lem}

\begin{proof}
We need only to prove that $(\huaR_A^\ast,-L_{\prec_A}^\ast,-R_{\succ_A}^\ast,\huaL_A^\ast)$ satisfy nine equations in Definition \ref{action}. By Lemma \ref{inv}, for all $x,y\in A,~~\xi,\eta\in A^*,$ we have
\begin{eqnarray*}
\langle \huaL_A^\ast(x)(\xi \succ_+ \eta)-\xi \succ_+ (\huaL_A^\ast(x)\eta),y  \rangle
&=& \langle L_{\prec_A}^*(I\eta) \xi, x\ast y \rangle-\langle  L_{\prec_A}^*(I(\huaL_A^*(x)\eta))\xi,y  \rangle\\
&=& \langle \xi, I\eta \prec_A (x\ast y)-I(\huaL_A^*(x)\eta)\prec_A y \rangle\\
&=& \langle \xi, I\eta \prec_A (x\ast y)-(I\eta \prec_A x)\prec_A y \rangle\\
&=& 0,
\end{eqnarray*}
which implies that $\huaL_A^\ast(x)(\xi \succ_+ \eta)=\xi \succ_+ (\huaL_A^\ast(x)\eta)$, for all $x\in A,~~\xi,\eta\in A^*.$ Similarly, the other eight equations are satisfied by direct calculations. Thus,
$(\huaR_A^\ast,-L_{\prec_A}^\ast,-R_{\succ_A}^\ast,\huaL_A^\ast)$ is an action of the dendriform algebra $(A,\prec_A,\succ_A)$ on the dendriform algebra $(A^\ast,\prec_+,\succ_+)$.
\end{proof}

\begin{thm}
Let $(A,A_r^*)$ be a quasi-triangular dendriform D-bialgebra induced by $r\in A\otimes A$. Then $r_+:(A^{\ast},\prec_+,\succ_+) \to (A,\prec_A,\succ_A)$ is a relative Rota-Baxter operator of weight $1$ with respect to the action $(\huaR_A^\ast,-L_{\prec_A}^\ast,-R_{\succ_A}^\ast,\huaL_A^\ast)$.
\end{thm}

\begin{proof}
By Theorem \ref{heart2} and \eqref{succr2}-\eqref{precr2}, for all $\xi,\eta\in A^*$, we have
\begin{eqnarray*}
r_+(\xi) \succ_A r_+(\eta)&=& r_+( \xi \succ_r \eta)\\
&=& r_+(\huaR_A^* (r_+(\xi))\eta-L_{\prec_A}^*(r_{+}(\eta))\xi+L_{\prec_A}^*(I\eta)\xi)\\
&=& r_+(\huaR_A^* (r_+(\xi))\eta-L_{\prec_A}^*(r_{+}(\eta))\xi+\xi \succ_+ \eta),
\end{eqnarray*}
and
\begin{eqnarray*}
r_+(\xi) \prec_A r_+(\eta)&=& r_+( \xi \prec_r \eta)\\
&=& r_+(\huaL_A^* (r_+(\eta)) \xi-R_{\succ_A}^*(r_{+}(\xi))\eta+R_{\succ_A}^*(I\xi)\eta)\\
&=& r_+(\huaL_A^* (r_+(\eta)) \xi-R_{\succ_A}^*(r_{+}(\xi))\eta+\xi \prec_+ \eta),
\end{eqnarray*}
which implies that $r_+:(A^{\ast},\prec_+,\succ_+) \to (A,\prec_A,\succ_A)$ is a relative Rota-Baxter operator of weight $1$ with respect to the action $(\huaR_A^\ast,-L_{\prec_A}^\ast,-R_{\succ_A}^\ast,\huaL_A^\ast)$ by Definition \ref{rRBo}.
\end{proof}

\begin{rmk}
In fact, if we define new multiplications $\prec_-,\succ_-$ on $A^{\ast}$ as follows:
$$\xi \succ_- \eta=\huaR_A^*(I\xi)\eta;\quad \xi \prec_- \eta=\huaL_A^*(I\eta)\xi, \quad \forall \xi,\eta\in A^*,$$
then $(A^*,\prec_-,\succ_-)$ is also a dendriform algebra. We can also prove that $r_-:(A^*,\prec_-,\succ_-) \to (A,\prec_A,\succ_A)$ is a relative Rota-Baxter operator of weight $1$ with respect to the action $(\huaR_A^\ast,-L_{\prec_A}^\ast,-R_{\succ_A}^\ast,\huaL_A^\ast)$.
\end{rmk}

\section{Rota-Baxter characterization of factorizable dendriform D-bialgebras}\label{QRBD}

In this section, first we introduce the notion of quadratic Rota-Baxter dendriform algebras. Then we show that there is a one-to-one correspondence between quadratic dendriform algebras and associative algebras with nondegenerate Connes cocycles. Then we add Rota-Baxter operators on these algebraic structures, and show that there is still a one-to-one correspondence between quadratic Rota-Baxter dendriform algebras and Rota-Baxter associative algebras with nondegenerate Connes cocycles.

\subsection{Quadratic Rota-Baxter dendriform algebras}

\begin{defi}\label{quadratic RB dendriform algebra}
	Let $(A,\prec_A,\succ_A)$ be a dendriform algebra and $\omega\in \wedge^{2} A^{\ast}$ a nondegenerate skew-symmetric bilinear form. If $\omega$ is invariant, i.e.
	\begin{eqnarray}\label{invariant3}
		\omega(x\succ_A y,z)=-\omega(x,y\prec_A z)=\omega(y,z\ast x),\quad \forall x,y,z\in A,
	\end{eqnarray}
	then $(A,\prec_A,\succ_A,\omega)$  is called a {\bf quadratic dendriform algebra.}
\end{defi}

Recall from \cite{Bai,Connes} that a skew-symmetric bilinear form $\omega\in\wedge^2 A^*$ is called a cyclic 2-cocycle in the sense of Connes, or simply a {\bf Connes cocycle} on an associative algebra $(A,\ast)$ if
$$\omega(x\ast y,z)+\omega(y\ast z,x)+\omega(z\ast x,y)=0,\quad  \forall x,y,z \in A.$$

There is a one-to-one correspondence between quadratic dendriform algebras and associative algebras with nondegenerate Connes cocycles.
\begin{thm}\label{lem:den-ass}{\rm(\cite{Bai})}
Let $(A,\prec_A,\succ_A,\omega)$ be a quadratic dendriform algebra. Then $(A,\ast,\omega)$ is an associative algebra with a nondegenerate Connes cocycle $\omega$, where the associative multiplication $\ast$ is given by \eqref{ast}.

Conversely, let $(A,\ast,\omega)$ be an associative algebra with a nondegenerate Connes cocycle $\omega$. Then there exists a compatible dendriform algebra structure $\prec_A,\succ_A$ on $A$ given by
\begin{eqnarray}
\label{eq:den-ass1}\omega(x\succ_A y,z)&=&\omega(y,z\ast x),\\
\label{eq:den-ass2}\omega(x\prec_A y,z)&=&\omega(x,y\ast z), \quad \forall x,y,z\in A,                               \end{eqnarray}
whose sub-adjacent associative algebra is exactly $(A,\ast)$. Moreover, $(A,\prec_A,\succ_A,\omega)$ is a quadratic dendriform algebra.
\end{thm}

Let $(A,\prec_A,\succ_A,P)$ be a Rota-Baxter dendriform algebra of weight $\lambda$. Then there are new dendriform multiplications $\prec_{P},\succ_{P}$ on $A$ defined by
$$ x \prec_{P} y= P(x)\prec_A y+x \prec_A P(y)+\lambda x\prec_A y; $$
$$ x \succ_{P} y= P(x)\succ_A y+x \succ_A P(y)+\lambda x\succ_A y. $$
The dendriform algebra $(A,\prec_{P},\succ_{P})$ is called the {\bf descendent dendriform algebra} and denoted by $A_{P}$. Furthermore, $P$ is a dendriform algebra homomorphism from the dendriform algebra $(A,\prec_{P},\succ_{P})$ to $(A,\prec_A,\succ_A)$.

We add Rota-Baxter operators on quadratic dendriform algebras and associative algebras with nondegenerate Connes cocycles, and introduce the notions of quadratic Rota-Baxter dendriform algebras and Rota-Baxter associative algebras with nondegenerate Connes cocycles.
\begin{defi}
	Let	$(A,\prec_A,\succ_A,P)$ be a Rota-Baxter dendriform algebra of weight $\lambda$ and $(A,\prec_A,\succ_A,\omega)$  a quadratic dendriform algebra. Then the tuple $(A,\prec_A,\succ_A,P,\omega)$ is called a {\bf quadratic Rota-Baxter dendriform algebra of weight $\lambda$} if the following compatibility condition holds:
	\begin{eqnarray}\label{compat condition}
		\omega(Px,y)+	\omega(x,Py)+\lambda \omega(x,y)=0, \quad \forall x,y \in A.
	\end{eqnarray}
\end{defi}
\begin{defi}
	Let $(A,\ast,P)$ be a Rota-Baxter associative algebra of weight $\lambda$ and $(A,\ast,\omega)$ an associative algebra with a nondegenerate Connes cocycle $\omega$. Then $(A,\ast,P,\omega)$ is  called a {\bf Rota-Baxter associative algebra with a nondegenerate Connes cocycle of weight $\lambda$} if $P$ and $\omega$ satisfy the compatibility condition \eqref{compat condition}.
\end{defi}

We now show that the relations between quadratic dendriform algebras and associative algebras with nondegenerate Connes cocycles given in Theorem \ref{lem:den-ass} also hold for quadratic Rota-Baxter dendriform algebras and Rota-Baxter associative algebras with nondegenerate Connes cocycles of the same weight.

\begin{thm}\label{thm:den-assconnes}
Let $(A,\prec_A,\succ_A,P,\omega)$ be a quadratic Rota-Baxter dendriform algebra of weight $\lambda$. Then  $(A,\ast,P,\omega)$ is a Rota-Baxter associative algebra with a nondegenerate Connes cocycle $\omega$ of weight $\lambda$, where $\ast$ is given by \eqref{ast}.

Conversely, let $(A,\ast,P,\omega)$ be a Rota-Baxter associative algebra with a nondegenerate Connes cocycle of weight $\lambda$. Then $(A,\prec_A,\succ_A,P,\omega)$ is a quadratic Rota-Baxter dendriform algebra of weight $\lambda$, where dendriform algebra multiplications $\prec_A,\succ_A$ are given by \eqref{eq:den-ass1} and \eqref{eq:den-ass2}.
\end{thm}
\begin{proof}
It is straightforward to deduce that if $P:A \rightarrow A$ is a  Rota-Baxter operator of weight $\lambda$ on a dendriform algebra $(A,\prec_A,\succ_A)$, then $P$  is a Rota-Baxter operator of weight $\lambda$ on the sub-adjacent associative algebra $(A,\ast)$. Since  $(A,\prec_A,\succ_A,\omega)$ is a quadratic dendriform algebra, by Theorem \ref{lem:den-ass}, $(A,\ast,\omega)$ is an associative algebra with a nondegenerate Connes cocycle $\omega$. Therefore, if $(A,\prec_A,\succ_A,P,\omega)$ is a quadratic Rota-Baxter dendriform algebra of weight $\lambda$, then $(A,\ast,P,\omega)$ is a Rota-Baxter associative algebra with a nondegenerate Connes cocycle $\omega$ of weight $\lambda$.

Conversely, let $(A,\ast,P,\omega)$ be a Rota-Baxter associative algebra with a nondegenerate Connes cocycle $\omega$ of weight $\lambda$. First by Theorem \ref{lem:den-ass}, $(A,\prec_A,\succ_A,\omega)$ is a quadratic dendriform algebra. In the following, we show that $P$ is a Rota-Baxter operator of weight $\lambda$ on the dendriform algebra $(A,\prec_A,\succ_A)$. By \eqref{eq:den-ass1}-\eqref{compat condition}, we have
\begin{eqnarray*}
		&&\omega \Big(P(x)\succ_A P(y)-P(P(x)\succ_A y+x \succ_A P(y)+\lambda x \succ_A y),z\Big)\\
		&=& \omega \big(P(x)\succ_A P(y),z \big)+\omega \big(P(x)\succ_A y+x\succ_A P(y)+\lambda x\succ_A y,P(z) \big)\\
        &&+\lambda \omega \big(P(x)\succ_A y+x\succ_A P(y)+\lambda x\succ_A y,z \big)\\
		&=& \omega \big(P(y),z \ast P(x) \big)+\omega \big(y,P(z)\ast P(x) \big)+\omega\big (P(y),P(z)\ast x \big)+\lambda\omega \big(y,P(z)\ast x \big)\\
		&&+\lambda\omega\big(y,z\ast P(x)\big)+\lambda\omega\big(P(y),z\ast x\big)+\lambda^2 \omega(y,z\ast x)\\
		&=& -\omega\big(y,P(z\ast P(x)) \big)-\lambda\omega \big(y,z\ast P(x) \big)+\omega \big(y,P(z)\ast P(x) \big)-\omega \big(y,P(P(z)\ast x) \big)-\lambda\omega \big(y,P(z)\ast x \big)\\
&&+\lambda\omega \big(y,P(z)\ast x \big)+\lambda\omega \big(y,z\ast P(x) \big)-\lambda\omega \big(y,P(z\ast x) \big)-\lambda^2\omega(y,z\ast x)+\lambda^2\omega(y,z\ast x)\\
		&=& \omega\Big(y,P(z)\ast P(x)-P(z\ast P(x)+P(z)\ast x+\lambda z\ast x)\Big)\\
		&=&0.
	\end{eqnarray*}
The last equality follows from the fact that $P$ is a Rota-Baxter operator of weight $\lambda$ on the associative algebra $(A,\ast)$. Furthermore, by the nondegeneracy of $\omega$, we have
$$ P(x)\succ_A P(y)=P(P(x)\succ_A y+x \succ_A P(y)+\lambda x \succ_A y),\quad \forall x,y,z\in A.  $$
Similarly, we also have
$$ P(x)\prec_A P(y)=P(P(x)\prec_A y+x \prec_A P(y)+\lambda x \prec_A y),\quad \forall x,y,z\in A,  $$
which implies that $P$ is a Rota-Baxter operator of weight $\lambda$ on the dendriform algebra $(A,\prec_A,\succ_A)$. Therefore, $(A,\prec_A,\succ_A,P,\omega)$ is a quadratic Rota-Baxter dendriform algebra of weight $\lambda$.
\end{proof}

\begin{ex}
Let $(A,\prec_A,\succ_A)$ be a dendriform algebra. Then $(A\ltimes A^*,\circ_c,\omega)$ is an  associative algebra with a nondegenerate Connes cocycle $\omega$, where $A\ltimes A^*$ is the semidirect product of the associative algebra $(A,\ast)$ and $A^*$, in which $(A^*;R_{\prec_A}^*,L_{\succ_A}^*)$ is a representation of the associative algebra $(A,\ast)$ \cite{Bai}. More precisely, the associative multiplication $\circ_c$ is given by
$$ (x+\xi) \circ_c (y+\eta)=x\ast y+R_{\prec_A}^*(x)\eta+L_{\succ_A}^*(y)\xi, \quad \forall x,y\in A,\xi,\eta \in A^*, $$
and the nondegenerate Connes cocycle $\omega$ is given by
$$ \omega(x+\xi,y+\eta)=\langle \xi,y \rangle-\langle \eta,x \rangle, \quad \forall x,y\in A,\xi,\eta \in A^*. $$
Moreover, for any $\lambda$, define linear maps $P_1,P_2:A\ltimes A^* \to A\ltimes A^*$ by
$$ P_1(x+\xi)=-\lambda x,\quad P_2(x+\xi)=-\lambda \xi,  \quad \forall x\in A,\xi \in A^*.   $$
Then $P_1,P_2$ are Rota-Baxter operators of weight $\lambda$ on $A\ltimes A^*$.
For any $i=1,2,$ it is straightforward to check that the following equation holds:
$$ \omega(P_i (x+\xi),y+\eta  )+\omega(x+\xi,P_i(y+\eta) )+\lambda\omega(x+\xi,y+\eta)=0. $$
Therefore, $(A\ltimes A^*,\circ_c,P_i,\omega)$ is a Rota-Baxter associative algebra with a nondegenerate Connes cocycle $\omega$ of weight $\lambda$.
\end{ex}

\begin{ex}\label{2-dimension}
Let $(A,\ast)$ be a 2-dimensional associative algebra with a basis $\{e_1,e_2\}$ whose nonzero multiplications are given as follows:
$$ e_1\ast e_1=e_1, \quad  e_1\ast e_2=e_2.$$
Let $\{e_1^*,e_2^*\}$ be the dual basis of $(A,\ast)$. Then $(A,\ast,\omega,P)$ defined by
$$\omega=e_1^* \wedge e_2^*,\quad P(e_1)=0,\quad P(e_2)=-\lambda e_2$$
is a Rota-Baxter associative algebra with a nondegenerate Connes cocycle $\omega$ of weight $\lambda$.
Moreover, the nonzero dendriform multiplications of the corresponding quadratic Rota-Baxter dendriform algebra $(A,\prec_A,\succ_A,P,\omega)$ of weight $\lambda$ are given by
\begin{eqnarray*}
 e_1 \succ_A e_2&=&e_2,\quad e_2 \succ_A e_1=-e_2,\\
 e_1 \prec_A e_1&=&e_1,\quad e_2 \prec_A e_1=e_2.\\
\end{eqnarray*}\end{ex}

\subsection{Factorizable dendriform D-bialgebras, quadratic Rota-Baxter dendriform algebras}

The following theorem shows that there is a one-to-one correspondence between a factorizable dendriform D-bialgebra and a quadratic Rota-Baxter dendriform algebra.
\begin{thm}\label{Factorizable dendriform algebra}
Let $(A,A_{r}^{\ast})$ be a factorizable dendriform D-bialgebra with $I=r_{+}-r_{-}$. Then $(A,P,\omega_{I})$ is a quadratic Rota-Baxter dendriform algebra of weight $\lambda$, where the linear map $P:A\longrightarrow A$ and $\omega_{I}\in \wedge^{2} A^{\ast}$ are defined by
\begin{eqnarray}
\label{P2}P&=&\lambda r_{-}\circ I^{-1},\\
\label{omegaI}\omega_I(x,y)&=&\langle I^{-1}x,y \rangle, \quad \forall x,y \in A.
\end{eqnarray}
Conversely, let $(A,\prec_A,\succ_A,P,\omega)$ be a quadratic Rota-Baxter dendriform algebra of weight $\lambda$ $(\lambda\neq 0),$ and $ \huaJ_{\omega}: A^{\ast}\longrightarrow A $ the induced linear isomorphism given by $\langle \huaJ_{\omega}^{-1}x,y \rangle :=\omega(x,y).$ Then $r \in A\otimes A $ defined by
	$$ r_{+}:=\frac{1}{\lambda} (P+\lambda \Id)\circ \huaJ_{\omega}:A^{\ast} \longrightarrow A, \quad r_{+}(\xi)=r(\xi,\cdot), \quad \forall \xi\in A^{\ast} $$
satisfies the D-equation $r_{12}\ast r_{13}-r_{13}\prec r_{23}-r_{23}\succ r_{12}=0$ and gives rise to a factorizable dendriform D-bialgebra $(A,A_{r}^{\ast})$.
\end{thm}

\begin{proof}
Since $r_{+},r_{-}:(A^{\ast},\prec_{r},\succ_{r}) \longrightarrow (A,\prec_A,\succ_A)$ are both dendriform algebra homomorphisms, for all $x,y\in A$, we have
\begin{eqnarray*}
I(I^{-1}x \succ_{r} I^{-1}y )&=&(r_{+}-r_{-})(I^{-1}x \succ_{r} I^{-1}y )\\
\nonumber&=&((I+r_{-})I^{-1}x) \succ_{A} ((I+r_{-})I^{-1}y)- (r_{-}I^{-1}x)\succ_{A} (r_{-}I^{-1}y)\\
\nonumber&=& (r_{-}I^{-1}x) \succ_{A} y+x \succ_{A} (r_{-}I^{-1}y)+x \succ_{A} y.
\end{eqnarray*}
Therefore, we have
\begin{eqnarray*}
P(P(x)\succ_{A} y+x \succ_{A} P(y)+\lambda x \succ_{A} y)&=& \lambda^{2} r_{-}I^{-1}\Big((r_{-}I^{-1}x) \succ_{A} y+x \succ_{A} (r_{-}I^{-1}y)+x \succ_{A} y \Big)\\
&=& \lambda^{2}r_{-}(I^{-1}x \succ_{r} I^{-1}y )\\
&=& \lambda^{2}(r_{-}I^{-1}x \succ_{A} r_{-}I^{-1}y )\\
&=& P(x) \succ_{A} P(y).
\end{eqnarray*}
Similarly, we also have
$$ P(P(x)\prec_{A} y+x \prec_{A} P(y)+\lambda x \prec_{A} y)=P(x) \prec_{A} P(y), $$
which implies that $P$ is a Rota-Baxter operator of weight $\lambda$ on the dendriform algebra $(A,\prec_A,\succ_A)$.

Next we show that $(A,P,\omega_{I})$ is a quadratic Rota-Baxter dendriform algebra. Since $I^{\ast}=-I$, we have
$$ \omega_{I}(x,y)=\langle I^{-1}x,y \rangle=-\langle x,I^{-1}y \rangle=-\omega_{I}(y,x),$$
 which means that $ \omega_{I}$ is skew-symmetric.

Since the skew-symmetric part $a$ of $r$ is $(L_\succ,R_\prec)$-invariant, by Lemma \ref{inv},
we have $I^{-1}\circ L_{\succ_A}(x)=\huaR_A^* (x) \circ I^{-1}$ and $\huaL_A^* (x) \circ I^{-1}=I^{-1}\circ R_{\prec_A}(x)$. Thus
\begin{eqnarray*}
\omega_{I}(x\succ_A y,z)-\omega_{I}(y,z\ast x)&=&\langle I^{-1}(x\succ_A y),z  \rangle-\langle I^{-1}(y),z\ast x \rangle\\
&=&\langle I^{-1}\circ L_{\succ_A}(x)(y)-\huaR_A^* \circ I^{-1}(y),z \rangle\\
&=&0,\\
\omega_{I}(x,y\prec_A z)+\omega_{I}(y,z\ast x)&=&\langle I^{-1}(x),y\prec_A z  \rangle+\langle I^{-1}(y),z\ast x \rangle\\
&=&-\langle x,I^{-1}\circ R_{\prec_A}(z)(y) \rangle+\langle x,\huaL_A^* (z) \circ I^{-1}(y) \rangle\\
&=&0,
\end{eqnarray*}
which implies that \eqref{invariant3} holds.

Moreover, by using $r^{\ast}_{-}=r_{+}$ and $I=r_{+}-r_{-}$, we have
\begin{eqnarray*}
&&\omega_{I}(x,Py)+\omega_{I}(Px,y)+\lambda \omega_{I}(x,y)\\
&=&\lambda\Big(\langle I^{-1}(x), r_{-}\circ I^{-1}(y) \rangle+\langle  I^{-1}\circ r_{-}\circ I^{-1}(x),y \rangle+\langle I^{-1}x,y \rangle \Big)\\
&=& \lambda\langle(-I^{-1}\circ r_{+}\circ I^{-1}+I^{-1}\circ r_{-}\circ I^{-1}+ I^{-1})(x),y \rangle\\
&=& 0,
\end{eqnarray*}
which implies that \eqref{compat condition} holds. Therefore, $(A,P,\omega_{I})$ is a quadratic Rota-Baxter dendriform algebra of weight $\lambda$.

Conversely, since $\omega$ is skew-symmetric, we have $\huaJ_{\omega}=-\huaJ_{\omega}^{\ast}$. By the fact that $\omega(x,Py)+\omega(Px,y)+\lambda \omega(x,y)=0,$ we have
	$$ \langle \huaJ_{\omega}^{-1}x,P(y)\rangle+\langle \huaJ_{\omega}^{-1}\circ P(x),y \rangle+\lambda\langle \huaJ_{\omega}^{-1}x,y \rangle=0, $$
	which implies that $P^{\ast}\circ \huaJ_{\omega}^{-1}+\huaJ_{\omega}^{-1}\circ P+\lambda \huaJ_{\omega}^{-1}=0, $ and then
	$$ \huaJ_{\omega}\circ P^{\ast}+P\circ \huaJ_{\omega}+\lambda \huaJ_{\omega}=0.$$
Thus we have
	$$ r_{-}:=r_{+}^{\ast}=\frac{1}{\lambda}(-\huaJ_{\omega}\circ P^{\ast}-\lambda \huaJ_{\omega})=\frac{1}{\lambda} P\circ \huaJ_{\omega}, $$
	and $\huaJ_{\omega}=r_{+}-r_{-}$. Define multiplications $\prec_r,\succ_r$ on $A^{\ast}$ by
\begin{eqnarray*}
\xi \succ_r \eta&=&\huaR_A^* (r_+(\xi)) \eta-L_{\prec_A}^*(r_{-}(\eta))\xi;\\
\xi \prec_r \eta&=&\huaL_A^* (r_+(\eta)) \xi-R_{\succ_A}^*(r_{-}(\xi))\eta, \quad \forall \xi,\eta\in A^*.
\end{eqnarray*}
By the fact that $\omega(x\succ_A y,z)=-\omega(x,y\prec_A z)=\omega(y,z\ast x)$, we have
\begin{eqnarray*}
\langle L_{\succ_A}^*(x)\circ \huaJ_{\omega}^{-1}(z),y \rangle&=&\langle \huaJ_{\omega}^{-1}\circ \huaR_A(x)z,y \rangle;\\
\langle R_{\prec_A}^*(z)\circ \huaJ_{\omega}^{-1}(x),y \rangle&=&\langle \huaJ_{\omega}^{-1}\circ \huaL_A(z)x,y \rangle,
\end{eqnarray*}
which implies that  $ \huaJ_{\omega}\circ R_{\prec_A}^*(x)=\huaL_A(x)\circ \huaJ_{\omega} $ and $ \huaJ_{\omega}\circ L_{\succ_A}^*(x)=\huaR_A(x)\circ \huaJ_{\omega}.$

Now we show that the following equation holds:
	\begin{eqnarray}\label{den alg iso1}
		\frac{1}{\lambda} \huaJ_{\omega}(\xi \succ_{r} \eta)=(\frac{1}{\lambda}\huaJ_{\omega}\xi)\succ_{P}(\frac{1}{\lambda}\huaJ_{\omega}\eta).
	\end{eqnarray}

On the one hand, we have
\begin{eqnarray*}
	\huaJ_{\omega}(\xi\succ_{r} \eta)&=&\huaJ_{\omega} \Big( \huaR_A^* (r_+(\xi)) \eta-L_{\prec_A}^*(r_{-}(\eta))\xi \Big)\\	&=&  L_{\succ_A}(r_+(\xi))\circ\huaJ_{\omega}(\eta)
-\huaJ_{\omega}\circ (L_{\prec_A}^{\ast}(r_{-}\eta))\xi\\
    &=& r_+(\xi)\succ_A \huaJ_{\omega}(\eta)-\huaJ_{\omega}\circ \Big((\huaL_A^{\ast}-L_{\succ_A}^{\ast}) (r_{-}\eta)\Big)\xi\\
	&=& r_+(\xi)\succ_A \huaJ_{\omega}(\eta)+ \huaJ_{\omega}(\xi)\succ_A r_-(\eta)\\
    &=& r_+(\xi)\succ_A (r_{+} - r_{-})(\eta)+ (r_{+} - r_{-})(\xi)\succ_A r_-(\eta)\\
	&=& r_{+}(\xi)\succ_A r_{+}(\eta)-r_{-}(\xi)\succ_A r_{-}(\eta).
\end{eqnarray*}
On the other hand, we have
\begin{eqnarray*}
	&&(\huaJ_{\omega}\xi)\succ_{P}(\huaJ_{\omega}\eta)\\
	&=&(P\huaJ_{\omega}\xi)\succ_A(\huaJ_{\omega}\eta)+(\huaJ_{\omega}\xi)\succ_A(P\huaJ_{\omega}\eta)
	+\lambda(\huaJ_{\omega}\xi)\succ_A(\huaJ_{\omega}\eta)\\
	&=& \lambda (r_{-}\xi)\succ_A(r_{+}\eta-r_{-}\eta)+\lambda(r_{+}\xi-r_{-}\xi)\succ_A (r_{-}\eta)+\lambda(r_{+}\xi-r_{-}\xi)\succ_A(r_{+}\eta-r_{-}\eta)\\
	&=& \lambda r_{+}(\xi)\succ_A r_{+}(\eta)-\lambda r_{-}(\xi)\succ_A r_{-}(\eta),
\end{eqnarray*}
which implies that  \eqref{den alg iso1} holds. Similarly, we can also have
	\begin{eqnarray}\label{den alg iso2}
		\frac{1}{\lambda} \huaJ_{\omega}(\xi \prec_{r} \eta)=(\frac{1}{\lambda}\huaJ_{\omega}\xi)\prec_{P}(\frac{1}{\lambda}\huaJ_{\omega}\eta).
	\end{eqnarray}
Thus $(A^{\ast},\prec_{r},\succ_{r})$ is a dendriform algebra and $\frac{1}{\lambda} \huaJ_{\omega}$ is a dendriform algebra isomorphism from $(A^{\ast},\prec_{r},\succ_{r})$ to $(A,\prec_{P},\succ_{P})$.

 Finally, by the fact that $P+\lambda\Id,P:(A,\prec_{P},\succ_{P})\longrightarrow (A,\prec_A,\succ_A)$ are both dendriform algebra homomorphisms, we deduce that
$$	 r_{+}:=\frac{1}{\lambda} (P+\lambda \Id)\circ \huaJ_{\omega},  \quad r_{-}:=\frac{1}{\lambda} P\circ \huaJ_{\omega}:(A^{\ast},\prec_{r},\succ_{r}) \longrightarrow (A,\prec_A,\succ_A)$$
	are both dendriform algebra homomorphisms. Therefore, by Theorem \ref{heart2}, $(A,A_{r}^{\ast})$ is a quasi-triangular dendriform D-bialgebra. Since $\huaJ_{\omega} = r_{+}-r_{-}$  is an isomorphism, the dendriform D-bialgebra $(A,A_{r}^{\ast})$ is factorizable.
\end{proof}

It is straightforward to check that if $P:A \longrightarrow A$ is a Rota-Baxter operator of weight $\lambda$ on a dendriform algebra $A$, then
\begin{eqnarray}
\widetilde{P}:=-\lambda\Id-P
\end{eqnarray}
is also a Rota-Baxter operator of weight $\lambda$.
\begin{cor}
Let $(A,A_{r}^{\ast})$ be a factorizable dendriform D-bialgebra with $I=r_{+}-r_{-}$. Then $(A,\widetilde{P},\omega_{I})$ is also a quadratic Rota-Baxter dendriform algebra of weight $\lambda$, where $\widetilde{P}=-\lambda\Id-P=-\lambda r_{+}\circ I^{-1}$ and $\omega_{I}\in \wedge^{2} A^{\ast}$ is defined by
\eqref{omegaI}.
\end{cor}
\begin{proof}
Since $(A,A_{r}^{\ast})$ is a factorizable dendriform D-bialgebra,  by Theorem \ref{Factorizable dendriform algebra},  $B$ satisfies the compatibility  condition \eqref{compat condition}.  Thus we have
\begin{eqnarray*}
&&\omega_{I}(x,\widetilde{P}y)+\omega_{I}(\widetilde{P}x,y)+\lambda \omega_{I}(x,y)\\
&=& -\omega_{I}(x,Py)-\lambda\omega_{I}(x,y)-\omega_{I}(Px,y)-\lambda \omega_{I}(x,y)+\lambda \omega_{I}(x,y)\\
&=& -\omega_{I}(x,Py)-\omega_{I}(Px,y)-\lambda\omega_{I}(x,y)\\
&=& 0.
\end{eqnarray*}
This implies that $(A,\widetilde{P},\omega_{I})$ is  a quadratic Rota-Baxter dendriform algebra of weight $\lambda$.
\end{proof}

By Theorem \ref{thm:den-assconnes} and Theorem \ref{Factorizable dendriform algebra}, we have
\begin{cor}\label{cor:assconnes-facden}
	Let $(A,\ast,P,\omega)$ be a Rota-Baxter associative algebra  with a nondegenerate Connes cocycle $\omega$ of weight $\lambda~(\lambda\ne 0)$, and $ \huaJ_\omega: A^{\ast}\longrightarrow A $ the induced linear isomorphism given by $\langle\huaJ_\omega^{-1}x,y \rangle :=\omega(x,y).$
 Then $r \in A\otimes A $ defined by
	$$ r_{+}:=\frac{1}{\lambda} (P+\lambda \Id)\circ \huaJ_\omega:A^{\ast} \longrightarrow A, \quad r_{+}(\xi)=r(\xi,\cdot), \quad \forall \xi\in A^{\ast}, $$
 gives rise to a factorizable dendriform D-bialgebra $(A,A_{r}^{\ast})$.
\end{cor}

By Corollary \ref{cor:assconnes-facden} and Example \ref{2-dimension}, we have

\begin{ex}
Let $(A,\ast)$ be a 2-dimensional associative algebra with a basis $\{e_1,e_2\}$ whose nonzero multiplications are given as follows:
$$ e_1\ast e_1=e_1, \quad  e_1\ast e_2=e_2.$$
Then $r=e_2\otimes e_1\in A\otimes A$ satisfies the D-equation and gives rise to a factorizable dendriform D-bialgebra $(A,A_{r}^{\ast})$.
\end{ex}

At the end of this subsection, we show that a quadratic Rota-Baxter dendriform algebra can give rise to an isomorphism from the regular representation to the coregular representation of a Rota-Baxter dendriform algebra. First, we introduce the definition of a representation of a Rota-Baxter dendriform algebra.

\begin{defi}
A {\bf representation of a Rota-Baxter dendriform algebra} $(A,\prec_A,\succ_A,P)$ of weight $\lambda$ on a vector space $V$ with respect to a linear transformation $T\in \gl(V)$ is a representation $(l_\succ,r_\succ,l_\prec,r_\prec)$ of the dendriform algebra $(A,\prec_A,\succ_A)$ on $V$, satisfying
\begin{eqnarray*}
l_\succ(Px)(Tu)&=&T\Big(l_\succ(Px)u+l_\succ(x)(Tu)+\lambda l_\succ(x)u\Big),\\
r_\succ(Px)(Tu)&=&T\Big(r_\succ(Px)u+r_\succ(x)(Tu)+\lambda r_\succ(x)u\Big),\\
l_\prec(Px)(Tu)&=&T\Big(l_\prec(Px)u+l_\prec(x)(Tu)+\lambda l_\prec(x)u\Big),\\
r_\prec(Px)(Tu)&=&T\Big(r_\prec(Px)u+r_\prec(x)(Tu)+\lambda r_\prec(x)u\Big),\quad\forall x\in A,u\in V.
\end{eqnarray*}
We will denote a representation by $(V,T,l_\succ,r_\succ,l_\prec,r_\prec)$.
\end{defi}

\begin{ex}
	Let $(A,\prec_A,\succ_A,P)$ be a Rota-Baxter dendriform algebra of weight $\lambda$. Then $(A,P,$
$L_{\succ_A},R_{\succ_A},L_{\prec_A},R_{\prec_A})$ is a representation of the Rota-Baxter dendriform algebra $(A,\prec_A,\succ_A,P)$, which is called the {\bf regular representation} of $(A,\prec_A,\succ_A,P)$.
\end{ex}

Let $(V,T,l_\succ,r_\succ,l_\prec,r_\prec)$ be a representation of a Rota-Baxter dendriform algebra $(A,\prec_A,\succ_A,P)$ of weight $\lambda$. Since $(l_\succ,r_\succ,l_\prec,r_\prec)$ is a representation of the dendriform algebra $A$ on $V$, we have the semidirect product dendriform algebra $A \ltimes V${\rm(\cite{Bai})}. Then define the map
$$ P\oplus T:A\ltimes V\rightarrow A\ltimes V, \quad x+u\mapsto Px+Tu. $$

\begin{pro}
With the above notations, $(A\ltimes V,P\oplus T)$ is a Rota-Baxter dendriform algebra of weight $\lambda$, called the semidirect product of $(A,\prec_A,\succ_A,P)$ and the representation $(V,T,l_\succ,r_\succ,l_\prec,r_\prec)$.
\end{pro}
\begin{proof}
	It follows by a direct calculation.
\end{proof}

\begin{defi}
Let $(V,T,l_\succ,r_\succ,l_\prec,r_\prec)$ and $(V',T',l'_\succ,r'_\succ,l'_\prec,r'_\prec)$ be two representations of a Rota-Baxter dendriform algebra  $(A,\prec_A,\succ_A,P)$ of weight $\lambda$. A {\bf homomorphism} from $(V,T,l_\succ,r_\succ,l_\prec,r_\prec)$ to $(V',T',l'_\succ,r'_\succ,l'_\prec,r'_\prec)$ is a linear map $\phi:V\longrightarrow V'$ such that for all $x\in A$, we have
\begin{equation*}
\left\{ \begin{aligned}
\phi\circ T&=T' \circ \phi;\quad &\\
\phi\circ l_\succ(x)&=l'_\succ(x)\circ \phi;\quad
&\phi\circ r_\succ(x)=r'_\succ(x)\circ \phi;\\
\phi\circ l_\prec(x)&=l'_\prec(x)\circ \phi;\quad
&\phi\circ r_\prec(x)=r'_\prec(x)\circ \phi.
\end{aligned}\right.
\end{equation*}
\end{defi}

Rota-Baxter dendriform algebras of weight $\lambda$ also have coregular representations.
\begin{thm}\label{coregular representation}
Let $(A,\prec_A,\succ_A,P)$ be a Rota-Baxter dendriform algebra of weight $\lambda$.
Then $(A^{\ast},-\lambda\Id$ $-P^{\ast},\huaR_A^{\ast},-L_{\prec_A}^{\ast},-R_{\succ_A}^{\ast},\huaL_A^{\ast})$ is a representation, which is called the {\bf coregular representation} of $(A,\prec_A,\succ_A,P)$.

Moreover, if $(A,\prec_A,\succ_A,P,\omega)$ is a quadratic Rota-Baxter dendriform algebra of weight $\lambda$, then the linear map $\omega^{\sharp}:A \longrightarrow A^{\ast} $ defined by $\langle \omega^{\sharp}(x),y \rangle=\omega(x,y) $ is an isomorphism from the regular representation $(A,P,L_{\succ_A},R_{\succ_A},L_{\prec_A},R_{\prec_A})$ to the coregular representation $(A^{\ast},-\lambda\Id-P^{\ast},\huaR_A^{\ast},-L_{\prec_A}^{\ast},$ $-R_{\succ_A}^{\ast},\huaL_A^{\ast})$.
\end{thm}

\begin{proof}
For all $\xi \in A^{\ast}$ and $x,y \in A$, since $P$ is a Rota-Baxter operator of weight $\lambda$ on the dendriform algebra $A$, we have
\begin{eqnarray*}
&&\langle \huaR_A^{\ast}(Px) \big((-\lambda\Id-P^{\ast})\xi \big)-(-\lambda\Id-P^{\ast})
\Big(\huaR_A^{\ast}(Px)\xi+ \huaR_A^{\ast}(x)\big((-\lambda\Id-P^{\ast})\xi\big) +\lambda \huaR_A^{\ast}(x) \xi \Big),y \rangle\\
&=& \langle \xi,Py \ast Px-P(y\ast Px+Py \ast x+\lambda y \ast x)  \rangle\\
&=&0;\\
&&\langle -L_{\prec_A}^{\ast}(Px) \big((-\lambda\Id-P^{\ast})\xi \big)-(-\lambda\Id-P^{\ast})
\Big(-L_{\prec_A}^{\ast}(Px)\xi-L_{\prec_A}^{\ast}(x)\big((-\lambda\Id-P^{\ast})\xi\big) -\lambda L_{\prec_A}^{\ast}(x) \xi \Big),y \rangle\\
&=& \langle \xi,-Px \prec_A Py+P(x\prec_A Py+Px \prec_A y+\lambda x \prec_A y)  \rangle\\
&=&0;\\
&&\langle -R_{\succ_A}^{\ast}(Px) \big((-\lambda\Id-P^{\ast})\xi \big)-(-\lambda\Id-P^{\ast})
\Big(-R_{\succ_A}^{\ast}(Px)\xi-R_{\succ_A}^{\ast}(x)\big((-\lambda\Id-P^{\ast})\xi\big) -\lambda R_{\succ_A}^{\ast}(x) \xi \Big),y \rangle\\
&=& \langle \xi,-Py \succ_A Px+P(y\succ_A Px+Py \succ_A x+\lambda y \succ_A x)  \rangle\\
&=&0;\\
&&\langle \huaL_A^{\ast}(Px) \big((-\lambda\Id-P^{\ast})\xi \big)-(-\lambda\Id-P^{\ast})
\Big(\huaL_A^{\ast}(Px)\xi+ \huaL_A^{\ast}(x)\big((-\lambda\Id-P^{\ast})\xi\big) +\lambda \huaL_A^{\ast}(x) \xi \Big),y \rangle\\
&=& \langle \xi,Px \ast Py-P(x\ast Py+Px \ast y+\lambda x \ast y)  \rangle\\
&=&0.
\end{eqnarray*}
Therefore, $(A^{\ast},-\lambda\Id-P^{\ast},\huaR_A^{\ast},-L_{\prec_A}^{\ast},-R_{\succ_ A}^{\ast},\huaL_A^{\ast})$ is a representation.

Let $(A,\prec_A,\succ_A,P,\omega)$ be a quadratic Rota-Baxter dendriform algebra of weight $\lambda$. By \eqref{invariant3} and the fact that $\omega$ is skew-symmetric, we have
\begin{equation*}
\left\{ \begin{aligned}
\omega^{\sharp}\circ L_{\succ_A}(x)&=\huaR_A^{\ast}(x) \circ \omega^{\sharp};\quad
&\omega^{\sharp}\circ R_{\succ_A}(x)=-L_{\prec_A}^{\ast}(x) \circ \omega^{\sharp};\\
\omega^{\sharp}\circ R_{\prec_A}(x)&=\huaL_A^{\ast}(x) \circ \omega^{\sharp};\quad
&\omega^{\sharp}\circ L_{\prec_A}(x)=-R_{\succ_A}^{\ast}(x) \circ \omega^{\sharp}.
\end{aligned}\right.
\end{equation*}
Moreover, by \eqref{compat condition}, we also have
$$ \omega^{\sharp}\circ P=(-\lambda\Id-P^{\ast}) \circ \omega^{\sharp}. $$
Note that $ \omega^{\sharp}$ is a linear isomorphism. Therefore, $ \omega^{\sharp}:A\longrightarrow A^*$ is an isomomorphism from the regular representation $(A,P,L_{\succ_A},R_{\succ_A},L_{\prec_A},R_{\prec_ A})$ to the coregular representation $(A^{\ast},-\lambda\Id-P^{\ast},\huaR_A^{\ast},-L_{\prec_ A}^{\ast},-R_{\succ_A}^{\ast},\huaL_A^{\ast})$.
\end{proof}

\end{document}